\documentclass[lineno]{biometrika}
\usepackage{xr}
\usepackage{enumerate}
\usepackage{subcaption}
\usepackage{amsmath}
\usepackage[all,import]{xy}

\usepackage{times}
\usepackage{bm}
\usepackage{natbib}
\usepackage{graphicx}

\usepackage[plain,noend]{algorithm2e}

\makeatletter
\renewcommand{\algocf@captiontext}[2]{#1\algocf@typo. \AlCapFnt{}#2} 
\def\@algocf@capt@plain{top}
\renewcommand{\algocf@makecaption}[2]{%
  \addtolength{\hsize}{\algomargin}%
  \sbox\@tempboxa{\algocf@captiontext{#1}{#2}}%
  \ifdim\wd\@tempboxa >\hsize
    \hskip .5\algomargin%
    \parbox[t]{\hsize}{\algocf@captiontext{#1}{#2}}
  \else%
    \global\@minipagefalse%
    \hbox to\hsize{\box\@tempboxa}
  \fi%
  \addtolength{\hsize}{-\algomargin}%
}
\makeatother

\def\Bka{{\it Biometrika}}

\def\ind{\begin{picture}(9,8)
         \put(0,0){\line(1,0){9}}
         \put(3,0){\line(0,1){8}}
         \put(6,0){\line(0,1){8}}
         \end{picture}
        }
\def\nind{\begin{picture}(9,8)
         \put(0,0){\line(1,0){9}}
         \put(3,0){\line(0,1){8}}
         \put(6,0){\line(0,1){8}}
         \put(1,0){{\it /}}
         \end{picture}
    }
    
\def\RD{\textsc{RD}}
\def\RR{\textsc{RR}}
\def\OR{\textsc{OR}}

\def\RRtrue{\textsc{RR}^{\textnormal{true}}}
\def\RRadj{\textsc{RR}^{\textnormal{adj}}}
\def\RRunadj{\textsc{RR}^{\textnormal{unadj}}}

\def\DCE{\textsc{DCE}}
\def\pr{\textup{pr}}
\def\E{E}
\def\true{\text{true}}

\def\FACE{\textsc{ACE}^{\textnormal{unadj}}}
\def\FDCE{\textsc{DCE}^{\textnormal{unadj}}}
\def\ACE{\textsc{ACE}}
\def\adj{{\textnormal{adj}}}
\def\true{{\textnormal{true}}}
\def\cov{\textnormal{cov}}

\renewcommand{\d}[1]{\ensuremath{\operatorname{d}\!{#1}}}

\begin{document}

%

\markboth{P. Ding, T. J. VanderWeele \and J. M. Robins}{Instrumental variables as bias amplifiers}

\title{Instrumental variables as bias amplifiers with general outcome and confounding}

\author{P. Ding}
\affil{Department of Statistics, University of California, Berkeley, California, USA.
\email{pengdingpku@berkeley.edu} } 

\author{T. J. VanderWeele \and J. M. Robins}
\affil{Departments of Epidemiology and Biostatistics, Harvard T. H. Chan School of Public Health, Boston, Massachusetts, USA. \email{tvanderw@hsph.harvard.edu} \email{robins@hsph.harvard.edu}  }

\maketitle

\begin{abstract}
Drawing causal inference with observational studies is the central pillar of many disciplines. One sufficient condition for identifying the causal effect is that the treatment-outcome relationship is unconfounded conditional on the observed covariates. It is often believed that the more covariates we condition on, the more plausible this unconfoundedness assumption is. This belief has had a huge impact on practical causal inference, suggesting that we should adjust for all pretreatment covariates. However, when there is unmeasured confounding between the treatment and outcome, estimators adjusting for some pretreatment covariate might have greater bias than estimators without adjusting for this covariate. This kind of covariate is called a bias amplifier, and includes instrumental variables that are independent of the confounder, and affect the outcome only through the treatment. Previously, theoretical results for this phenomenon have been established only for linear models. We fill in this gap in the literature by providing a general theory, showing that this phenomenon happens under a wide class of models satisfying certain monotonicity assumptions. We further show that when the treatment follows an additive or multiplicative model conditional on the instrumental variable and the confounder, these monotonicity assumptions can be interpreted as the signs of the arrows of the causal diagrams.

\end{abstract}


\section{Introduction}

Causal inference from observational data is an important but challenging problem for empirical studies in many disciplines. Under the potential outcomes framework \citep{Neyman::1923, Rubin::1974}, the causal effects are defined as comparisons between the potential outcomes under treatment and control, averaged over a certain population of interest. One sufficient condition for nonparametric identification of the causal effects is the ignorability condition \citep{Rosenbaum::1983}, that the treatment is conditionally independent of the potential outcomes given those pretreatment covariates that confound the relationship between the treatment and outcome. To make this fundamental assumption as plausible as possible, many researchers suggest that the set of collected pretreatment covariates should be as rich as possible. It is often believed that ``typically, the more conditional an assumption, the more generally acceptable it is'' \citep{rubin2009should}, and therefore ``in principle, there is little or no reason to avoid adjustment for a true covariate, a variable describing subjects before treatment'' \citep[][pp. 76]{rosenbaum2002observational}.

Simply adjusting for all pretreatment covariates \citep{d1998tutorial, rosenbaum2002observational, hirano2001estimation}, or the pretreatment criterion \citep{vanderweele2011new}, has a sound justification from the view point of design and analysis of randomized experiments. \citet{cochran1965planning}, citing \citet{dorn1953philosophy}, suggested that the planner of an observational study should always ask himself the question, ``How would the study be conducted if it were possible to do it by controlled experimentation?'' Following this classical wisdom, \citet{rubin2007design, rubin2008author, rubin2008objective, rubin2009should} argued that the design of observational studies should be in parallel with the design of randomized experiments, i.e., because we balance all pretreatment covariates in randomized experiments, we should also follow this pretreatment criterion and balance or adjust for all pretreatment covariates when designing observational studies.


However, this pretreatment criterion can result in increased bias under certain data generating processes.
We highlight two important classes of such data generating processes for which the pretreatment criterion may be problematic.
The first class is captured by an example of \citet{greenland1986identifiability}, in which conditioning on a pretreatment covariate invalidates the ignorability assumption and thus a conditional analysis is biased; yet the ignorability assumption holds unconditionally, so an analysis that ignores the covariate is unbiased. 
Several researchers have shown that this phenomenon is generic when the data are generated under the causal diagram in Figure \ref{fig:Mbias}. In this diagram, the ignorability assumption holds unconditionally but not conditionally \citep{Pearl2000, spirtes2000causation, Greenland::2003, Pearl2009a, Shrier2008, Shrier2009, Sjolander2009, ding2015adjust}. In Figure \ref{fig:Mbias}, a pretreatment covariate $M$ is associated with two independent unmeasured covariates $U$ and $U'$, but $M$ does not itself affect either the treatment $A$ or outcome $Y$. Because the corresponding causal diagram looks like the English letter M, this phenomenon is called M-Bias.

\begin{figure}[!t]
\centering
\begin{subfigure}[b]{0.42\linewidth}
$$
\begin{xy}
\xymatrix{
U\ar[dr] \ar[dd] & & U' \ar[dl] \ar[dd] \\
&M& \\
A\ar[rr]&& Y
}
\end{xy}
$$
\caption{Directed Acyclic Graph for M-Bias. $U$ and $U'$ are unobserved, and $M$ is observed.}
\label{fig:Mbias}
\end{subfigure} 
\quad \quad 
\begin{subfigure}[b]{0.42\linewidth}
$$
\begin{xy}
\xymatrix{
&&&&&\\
& & & U \ar[dl] \ar[dr] \\
Z \ar[rr]  & & A\ar[rr]  & & Y}
\end{xy}
$$
\caption{Directed Acyclic Graph for Z-Bias. $U$ is an unmeasured confounder and $Z$ is an instrumental variable for the treatment-outcome relationship.}
\label{fg::DAG}
\end{subfigure} 
\caption{Two Directed Acyclic Graphs. $A$ is the treatment, and $Y$ is the outcome of interest.}
\label{fig:two-eg}
\end{figure}
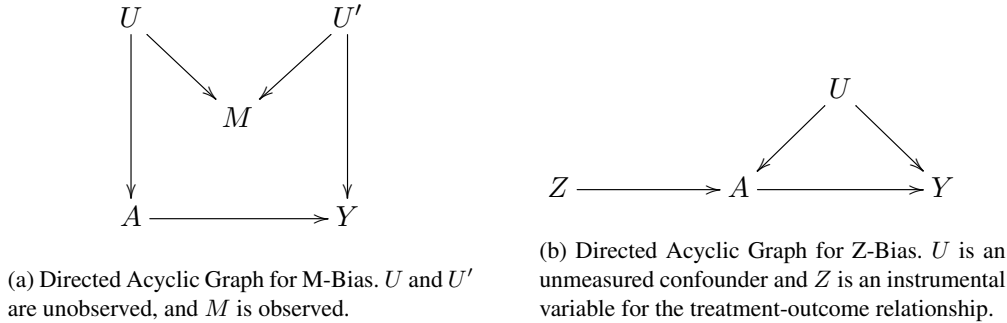


The second class of processes, which constitute the subject of this paper, are represented by the causal diagram in Figure \ref{fg::DAG}. Owing to confounding by the unmeasured common cause $U$ of the treatment $A$ and the outcome $Y$, both the analysis that adjusts and the analysis that fails to adjust for pretreatment measured covariates are biased. If the magnitude of the bias is larger when we adjust for a particular pretreatment covariate than when we do not, we refer to the covariate as a bias amplifier. Of particular interest is to determine the conditions under which an instrumental variable is a bias amplifier.
An instrumental variables is a pretreamtnet covariate that is independent of the confounder $U$ and has no direct effect on the outcome except through its effect on the treatment. The variable $Z$ in Figure \ref{fg::DAG} is an example.
\citet{heckman2004using} and \citet{bhattacharya2007instrumental} showed numerically that when the treatment and outcome are confounded, adjusting for an instrumental variable can result in greater bias than the unadjusted estimator. Wooldridge theoretically demonstrated this in linear models in a technical report in 2006, which was finally published as \citet{wooldridge2009should}. Because instrumental variables are often denoted by $Z$ as in Figure \ref{fg::DAG}, this phenomenon is called Z-Bias.

The treatment assignment is a function of the instrumental variable, the unmeasured confounder and some other independent random error, which are the three sources of variation of the treatment. If we adjust for the instrumental variable, the treatment variation is driven more by the unmeasured confounder, which could result in increased bias due to this confounder. Seemingly paradoxically, without adjusting for the instrumental variable, the observational study is more like a randomized experiment, and the bias due to confounding is smaller.
Although applied researchers \citep{myers2011effects, walker2013matching, brooks2013squeezing, ali2014propensity} have confirmed through extensive simulation studies that this bias amplification phenomenon exists in a wide range of reasonable models, definite theoretical results have been established only for linear models. We fill in this gap in the literature by showing that adjusting for an instrumental variable amplifies bias for estimating causal effects under a wide class of models satisfying certain monotonicity assumptions. When the instrumental variable and the confounder have either no additive or no multiplicative interaction on the treatment, these  assumptions can be interpreted as the signs of the arrows of the causal diagram \citep{vanderweele2010signed}. 
However, we also show that there exist data generating processes under which an instrumental variable is not a bias amplifier.

\section{Framework and Notation}
\label{sec::notation}

%

%
%
%

We consider a binary treatment $A$, an instrumental variable $Z$, an unobserved confounder $U,$ and an outcome $Y$, with the joint distribution depicted by the causal diagram in Figure \ref{fg::DAG}. Let $  \ind $ denote conditional independence between random variables. Then the instrumental variable $Z$ in Figure \ref{fg::DAG} satisfies $Z\ind U$, $Z\ind Y\mid (A,U)$ and $Z\nind A.$ 
We first discuss analysis conditional on observed pretreatment covariates $X$, and comment on averaging over $X$ in \S \ref{sec::discussion} and the Supplementary Material. 
We define the potential outcomes of $Y$ under treatment $a$ as $Y(a)$, $(a=1,0).$ The true average causal effect of $A$ on $Y$ for the population actually treated is
$$
\ACE_1^\true=\E\{  Y(1) \mid A=1\} - \E\{  Y(0) \mid A=1\},
$$
for the population who are actually in the control condition it is
$$
\ACE_0^\true=\E\{  Y(1) \mid A=0 \} - \E\{  Y(0) \mid A=0 \},
$$
and for the whole population it is
$$
\ACE^\true = \E\{  Y(1) \} - \E\{  Y(0) \} .
$$

Define $m_a(u) = \E(Y\mid A=a,U=u)$ to be the conditional mean of the outcome given the treatment and confounder. As illustrated by Figure \ref{fg::DAG}, because $U$ suffices to control confounding between $A$ and $Y$, the ignorability assumption $A\ind Y(a)\mid U$ holds for $a=0$ and $1$. Therefore, according to $Y=AY(1)+(1-A)Y(0)$, we have
\begin{eqnarray*}
\ACE_1^\true &=& \E(Y\mid A=1) - \int m_0(u) F(\d u\mid A=1), \\
\ACE_0^\true &=& \int m_1(u) F(\d u\mid A=0) - \E(Y\mid A=0), \\
\ACE^\true &=& \int m_1(u) F(\d u) - \int m_0(u) F(\d u).
\end{eqnarray*}

The unadjusted estimator is the naive comparison between the treatment and control means
$$
\FACE = \E(Y\mid A=1) - \E(Y\mid A=0).
$$

Define $\mu_a(z) = \E(Y\mid A=a, Z=z)$ as the conditional mean of the outcome given the treatment and instrumental variable.
Because the instrumental variable $Z$ is also a pretreatment covariate unaffected by the treatment, the usual strategy to adjust for all pretreatment covariates suggests using the adjusted estimator for the population under treatment
$$
\ACE_1^\adj = \E(Y\mid A=1) - \int   \mu_0(z) F(\d z\mid A=1),
$$
for the population under control
$$
\ACE_0^\adj = \int  \mu_1(z) F(\d z \mid A=0) - \E(Y\mid A=0),
$$
and for the whole population
$$
\ACE^\adj = \int  \mu_1(z) F(\d z)  -  \int   \mu_0(z) F(\d z) .
$$

Surprisingly, for linear structural equation models on $(Z,U,A,Y)$, previous theory demonstrated that the magnitudes of the biases of the adjusted estimators are no smaller than the unadjusted ones \citep{pearl2010class, pearl2011invited, pearl2013linear, wooldridge2009should}. The goal of the rest of our paper is to show that this phenomenon exists in more general scenarios.

\section{Scalar Instrumental Variable and Scalar Confounder}
\label{sec::scalar}

We first give a theorem for a scalar instrumental variable $Z$ and a scalar confounder $U.$

\begin{theorem}
\label{thm::Zbias-scalar-case-general}
In the causal diagram of Figure \ref{fg::DAG} with scalar $Z$ and $U$, if
\begin{enumerate}
[(a)]
\item
$\pr(A=1\mid Z=z)$ is non-decreasing in $z$, $\pr(A=1\mid U=u)$ is non-decreasing in $u$, and $\E(Y\mid A=a, U=u)$ is non-decreasing in $u$ for both $a=0$ and $1$;

\item
$\E(Y\mid A=a, Z=z)$ is non-increasing in $z$ for both $a=0$ and $1$,
\end{enumerate}
then
\begin{eqnarray}
\label{eq::result}
\begin{pmatrix}
\ACE_1^\adj \\
\ACE_0^\adj \\
\ACE^\adj
\end{pmatrix}
\geq 
\begin{pmatrix}
 \FACE \\
 \FACE \\
 \FACE
\end{pmatrix}
 \geq 
\begin{pmatrix}
\ACE_1^\true\\
\ACE_0^\true\\
\ACE^\true
\end{pmatrix}.
\end{eqnarray}
\end{theorem}

Inequalities among vectors as in \eqref{eq::result} should be interpreted as component-wise relationships. 
Intuitively, the monotonicity in Condition (a) of Theorem \ref{thm::Zbias-scalar-case-general} requires non-negative dependence structures on arrows $Z\rightarrow A$, $U\rightarrow A$ and $U\rightarrow Y$ in the causal diagram of Figure \ref{fg::DAG}. Because the dependence is in expectation, Condition (a) of Theorem \ref{thm::Zbias-scalar-case-general} is weaker than the requirement of signed directed acyclic graphs \citep{vanderweele2010signed}.

The monotonicity in Condition (b) of Theorem \ref{thm::Zbias-scalar-case-general} reflects the collider bias caused by conditioning on $A$. As noted by \citet{Greenland::2003}, in many cases, if $Z$ and $U$ affect $A$ in the same direction, then the collider bias caused by conditioning on $A$ is often in the opposite direction. Lemmas \ref{lemma::interactions}--\ref{lemma::inter-multi-general} in the Supplementary Material show that, if $Z$ and $U$ are independent and have non-negative additive or multiplicative effects on $A$, then conditioning on $A$ results in negative association between $Z$ and $U$. This negative collider bias, coupled with the positive association between $U$ and $Y$, further implies negative association between $Z$ and $Y$ conditional on $A$ as stated in Condition (b) of Theorem \ref{thm::Zbias-scalar-case-general}.

For easy interpretation, we will give sufficient conditions for Z-Bias which require no interaction of $Z$ and $U$ on $A.$ 
When $A$ given $Z$ and $U$ follows an additive model, we have the following theorem.

\begin{theorem}
\label{thm::Zbias-scalar-case}
In the causal diagram of Figure \ref{fg::DAG} with scalar $Z$ and $U$, \eqref{eq::result} holds if
\begin{enumerate}
[(a)]
\item
$\pr(A=1\mid Z=z, U=u) = \beta(z) + \gamma(u)$;

\item
$\beta(z)$ is non-decreasing in $z$, $\gamma(u)$ is non-decreasing in $u$, and $\E(Y\mid A=a, U=u)$ is non-decreasing in $u$ for both $a=1$ and $0$;

\item
the essential supremum of $U$ given $(A=a,Z=z)$ depends only on $a$.
\end{enumerate}
\end{theorem}

In summary, when $A$ given $Z$ and $U$ follows an additive model and monotonicity of Theorem \ref{thm::Zbias-scalar-case} holds, both unadjusted and adjusted estimators have non-negative biases for the true average causal effects for the treatment, control and the whole populations. Furthermore, the adjusted estimators, either for the treatment, control or the whole populations, have larger biases than the unadjusted estimator, i.e., Z-Bias arises.

When both the instrumental variable $Z$ and the confounder $U$ are binary, Theorem \ref{thm::Zbias-scalar-case} has an even more interpretable form. Define $p_{zu} = \pr(A=1\mid Z=z, U=u) $ for $z,u=0$ and $1.$

\begin{corollary}
\label{thm::Zbias-binary-case}
In the causal diagram of Figure \ref{fg::DAG} with binary $Z$ and $U$, \eqref{eq::result} holds if
\begin{enumerate}
[(a)]
\item 
there is no additive interaction of $Z$ and $U$ on $A$, i.e., $p_{11}-p_{10}-p_{01}+p_{00}=0;$

\item
$Z$ and $U$ have monotonic effects on $A$, i.e., $p_{11}\geq \max(p_{10}, p_{01})$ and $\min(p_{10}, p_{01}) \geq p_{00}$, and $\E(Y\mid A=a,U=1)\geq \E(Y\mid A=a,U=0)$ for both $a=1$ and $0.$
\end{enumerate}
\end{corollary}

When $A$ given $Z$ and $U$ follows an multiplicative model, we have the following theorem.

\begin{theorem}
\label{thm::Zbias-scalar-case-multi}
In the causal diagram of Figure \ref{fg::DAG} with scalar $Z$ and $U$, \eqref{eq::result} holds if we replace Condition (a) of Theorem \ref{thm::Zbias-scalar-case} by
\begin{enumerate}
\item[(a')]
$\pr(A=1\mid Z=z, U=u) = \beta(z) \gamma(u)$.
\end{enumerate}
\end{theorem}

When both the instrument $Z$ and the confounder $U$ are binary, Theorem \ref{thm::Zbias-scalar-case-multi} can be simplified.

\begin{corollary}
\label{thm::Zbias-binary-case-multi}
In the causal diagram of Figure \ref{fg::DAG} with binary $Z$ and $U$, \eqref{eq::result} holds if we replace Condition (a) of Corollary \ref{thm::Zbias-binary-case} by
\begin{enumerate}
\item[(a')]
there is no multiplicative interaction of $Z$ and $U$ on $A$, i.e., $p_{11}p_{00} = p_{10}p_{01} .$
\end{enumerate}
\end{corollary}

We invoke the assumptions of no additive and multiplicative interaction of $Z$ and $U$ on $A$ in Theorems \ref{thm::Zbias-scalar-case} and \ref{thm::Zbias-scalar-case-multi} for easy interpretation. They are sufficient but not necessary conditions for Z-Bias. In fact, we show in the proofs that Conditions (a) and (a') in Theorems \ref{thm::Zbias-scalar-case} and \ref{thm::Zbias-scalar-case-multi} and Corollaries \ref{thm::Zbias-binary-case} and \ref{thm::Zbias-binary-case-multi} can be replaced by weaker conditions. For the case with binary $Z$ and $U$, these conditions are particularly easy to interpret: 
\begin{eqnarray}
\label{eq::weaker-conditions}
{  p_{11}p_{00} \over p_{10}p_{01} }  \leq 1,\quad
{ (1-p_{11})(1-p_{00}) \over  (1-p_{10}) (1-p_{01}) } \leq 1,
\end{eqnarray}
i.e., $Z$ and $U$ have non-positive multiplicative interaction on both the presence and absence of $A.$ Even if Condition (a) or (a') does not hold, one can show that half of the parameter space of $(p_{11}, p_{10}, p_{01}, p_{00})$ satisfies the weaker condition \eqref{eq::weaker-conditions}, which is only sufficient, not necessary. Therefore, even in the presence of additive or multiplicative interaction, Z-Bias arises in more than half of the parameter space for binary $(Z,U,A,Y)$.

\section{General Instrumental Variable and General Confounder}
\label{sec::general}

When the instrumental variable $Z$ and the confounder $U$ are vectors, Theorems \ref{thm::Zbias-scalar-case-general}--\ref{thm::Zbias-scalar-case-multi} still hold if the monotonicity assumptions hold for each component of $Z$ and $U$, and $Z$ and $U$ are multivariate totally positive of order two \citep{karlin1980classes}, including the case that the components of $Z$ and $U$ are mutually independent \citep{esary1967association}. A random vector $W$ is multivariate totally positive of order two, if its density $f(\cdot)$ satisfies $f\{ \max (w_1, w_2)  \}  f\{ \min (w_1, w_2)  \} \geq f(w_1) f(w_2)  $, where $\max (w_1, w_2)$ and $\min (w_1, w_2) $ are component-wise maximum and minimum of the vectors $w_1$ and $w_2$. In the following, we will develop general theory for Z-Bias without the total positivity assumption about the components of $Z$ and $U.$

It is relatively straightforward to summarize a general instrumental variable $Z$ by a scalar propensity score $\Pi = \Pi(Z) = \pr(A=1\mid Z)$, because $Z\ind A\mid \Pi(Z)$ as shown in \citet{Rosenbaum::1983}. 
We define $\nu_a(\pi) = \E(Y\mid A=a, \Pi = \pi)$. The adjusted estimator for the population under treatment is
$$
\ACE_1^\adj = \E(Y\mid A=1) - \int   \nu_0(\pi) F(\d \pi \mid A=1),
$$
the adjusted estimator for the population under control is
$$
\ACE_0^\adj = \int  \nu_1(\pi) F(\d \pi \mid A=0) - \E(Y\mid A=0),
$$
and the adjusted estimator for the whole population is
$$
\ACE^\adj = \int  \nu_1(\pi) F(\d \pi)  -  \int   \nu_0(\pi) F(\d \pi)  .
$$
When $Z$ is scalar, then the above three formulas reduce to the ones in Section \ref{sec::scalar}.

\citet{greenland1986identifiability} showed that for the causal effect on the treated population, $Y(0)$ alone suffices to control for confounding; likewise, for the causal effect on the control population, $Y(1)$ alone suffices to control for confounding. If interest lies in all three of our average causal effects, then we need to take $U = \{Y(1), Y(0)\}$ as the ultimate confounder for the relationship of $A$ on $Y.$ This is not an assumption about $U$. Because $Y = AY(1)+(1-A)Y(0) $ is a deterministic function of $A$ and $\{Y(1), Y(0)\}$, this implies that $U = \{Y(1), Y(0)\}$ satisfies the ignorability assumption \citep{Rosenbaum::1983}, or blocks all the back-door paths from $A$ to $Y$ \citep{pearl1995causal, Pearl2000}. We represent the causal structure in Figure \ref{fg::DAG2}.

\begin{figure}[!t]
\centering
$$
\begin{xy}
\xymatrix{
& &&& & U = \{Y(1), Y(0)\} \ar[ddl] \ar[ddr] \\
\\
Z\ar[rr]&&\Pi = \Pi(Z) \ar[rr]  & & A\ar[rr]  & & Y}
\end{xy}
$$
\caption{Directed Acyclic Graph for Z-Bias With General Instrument and Confounder}\label{fg::DAG2}
\end{figure}
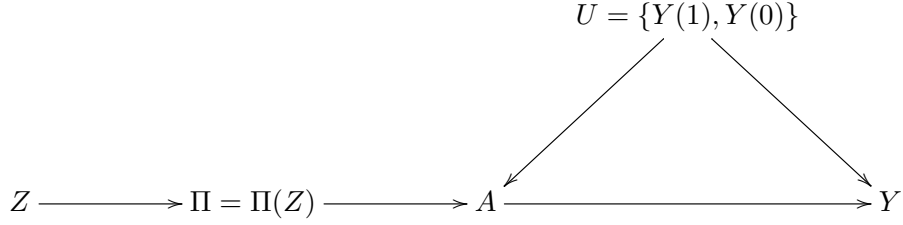

%

We first state a theorem without assuming the structure of the causal diagram in Figure \ref{fg::DAG2}.

\begin{theorem}
\label{thm::general-iv-conf-weak}
If for both $a=1$ and $0$, $\pr\{ A=1\mid Y(a) \}$ is non-decreasing in $Y(a)$, and $\cov\{ \Pi, \nu_a(\Pi) \} \leq 0$, then \eqref{eq::result} holds.
\end{theorem}

In a randomized experiment $A\ind Y(a)$, so the dependence of $\pr\{ A=1\mid Y(a) \}$ on $Y(a)$ characterizes the self-selection process of an observational study. The condition $\cov\{ \Pi, \nu_a(\Pi) \} \leq 0$ in Theorem \ref{thm::general-iv-conf-weak} is another measure of the collider-bias caused by conditioning on $A$, as $\nu_a(\pi) = \E\{ Y(a) \mid A=a, \Pi = \pi \}$ and $Y(a)$ is a component of $U$ in Figure \ref{fg::DAG2}. This  measure of collider bias is more general than the one in Theorem \ref{thm::Zbias-scalar-case-general}. Analogous to Section \ref{sec::scalar}, we will present more transparent sufficient conditions for Z-Bias to aid interpretation.

In the following, we use the distributional association measure \citep{cox2003general, ma2006collapsibility, xie2008some}, i.e., random variable $V$ has a non-negative distributional association on random variable $W$, if the conditional distribution satisfies $\partial F(w\mid v) / \partial v \leq 0$ for all $v$ and $w.$ If the random variables are discrete, then partial differentiation is replaced by differencing between adjacent levels \citep{cox2003general}.

If there is no additive interaction between $\Pi$ and $\{Y(1), Y(0)\}$ on $A$, then we have the following results.

\begin{theorem}
\label{thm::general-iv-conf}
In the causal diagram of Figure \ref{fg::DAG2}, \eqref{eq::result} holds if
\begin{enumerate}
[(a)]
\item
$\pr(A=1\mid \Pi, U) = \Pi + \delta\{  Y(1) \} + \eta\{  Y(0) \}$ with $\delta(\cdot)$ and $\eta(\cdot)$ being non-decreasing;

\item
$\{Y(1), Y(0)\}$ have non-negative distributional associations on each other, i.e., $\partial F(y_1 \mid y_0) / \partial y_0 \leq 0$ and $\partial F(y_0 \mid y_1) / \partial y_1 \leq 0$ for all $y_1$ and $y_0$;

\item
the essential supremum  of $Y(1)$ given $Y(0)$ does not depend on $Y(0)$, and the essential supremum  of $Y(0)$ given $Y(1)$ does not depend on $Y(1)$.
\end{enumerate}
\end{theorem}

\begin{remark}
If we impose an additive model 
$
\pr(A=1\mid \Pi, U) = h(\Pi) + \delta\{  Y(1) \} + \eta\{  Y(0) \},
$ 
then independence of $\Pi$ and $U$ implies that 
$
\pr(A=1\mid \Pi) = h(\Pi) + \E[  \delta\{  Y(1) \}  ]+ \E [\eta\{  Y(0) \} ] = \Pi.
$
Therefore, we must have $h(\Pi) = \Pi$ and $\E[  \delta\{  Y(1) \}  ]+ \E [\eta\{  Y(0) \} ] = 0. $ 
\end{remark}

When the outcome is binary, the distributional association between $Y(1)$ and $Y(0)$ becomes their odds ratio \citep{xie2008some}, and non-negative distributional association between $Y(1)$ and $Y(0)$ is equivalent to 
$$
\OR_Y 
=  \frac{   \pr\{  Y(1) = 1,  Y(0) = 1 \}  \pr\{  Y(1) = 0,  Y(0) = 0 \}        }
{ \pr\{  Y(1) = 1,   Y(0) = 0 \}   \pr\{  Y(1) = 0,   Y(0) = 1 \}   } \geq 1 .
$$
We can further relax the model assumption of $A$ given $\Pi$ and $U$ by allowing for non-negative interaction between $Y(1)$ and $Y(0)$ on $A.$

\begin{corollary}
\label{thm::general-iv-conf-binary}
In the causal diagram of Figure \ref{fg::DAG2} with a binary outcome $Y$, \eqref{eq::result} holds if
\begin{enumerate}
[(a)]
\item\label{condition::general-iv-conf-binary-a}
$\pr(A=1\mid \Pi, U) =\alpha +  \Pi + \delta  Y(1) + \eta   Y(0) + \theta Y(1)Y(0)$ with $\delta, \eta, \theta \geq 0$;

\item
$\OR_Y \geq 1.$
\end{enumerate}
\end{corollary}

\begin{remark}
If we have an additive model of $A$ given $\Pi$ and $U$,
$
\pr(A=1\mid \Pi, U) = h(\Pi) + g(U),
$
then the functional form $g(U) = \alpha +  \delta  Y(1) + \eta   Y(0) + \theta Y(1)Y(0)$ imposes no restriction for binary outcome. Furthermore, $\pr(A=1\mid \Pi) = \Pi$ implies that $h(\Pi) = \Pi$ and $E\{ g(U) \} = 0$, i.e., $\alpha = -\delta E\{  Y(1)\} -\eta \E\{ Y(0)  \} - \theta \E \{ Y(1)Y(0)  \}.$ Therefore, the additive model in Condition \eqref{condition::general-iv-conf-binary-a} of Corollary \ref{thm::general-iv-conf-binary} is
$$
\pr(A=1\mid \Pi, U) =  \Pi + \delta [Y(1) - E\{  Y(1)\}  ]   + \eta  [  Y(0)- \E\{ Y(0)  \} ]  
+ \theta[  Y(1)Y(0) -  \E \{ Y(1)Y(0)  \} ]  .
$$
\end{remark}

If there is no multiplicative interaction of $\Pi$ and $\{ Y(1), Y(0) \}$ on $Z$, then we have the following results.

\begin{theorem}
\label{thm::general-iv-conf-multi}
In the causal diagram of Figure \ref{fg::DAG2}, \eqref{eq::result} holds if we replace Condition (a) of Theorem \ref{thm::general-iv-conf} by
\begin{enumerate}
[(a')]
\item
$\pr(A=1\mid \Pi, U) = \Pi  \delta\{  Y(1) \}  \eta\{  Y(0) \}$ with $\delta(\cdot)$ and $\eta(\cdot)$ being non-decreasing.
\end{enumerate}
\end{theorem}

\begin{corollary}
\label{thm::general-iv-conf-binary-multi}
In the causal diagram of Figure \ref{fg::DAG2} with a binary outcome $Y$, \eqref{eq::result} holds if we replace Condition (a) of Corollary \ref{thm::general-iv-conf-binary} by
\begin{enumerate}
[(a')]
\item
$\pr(A=1\mid \Pi, U) = \alpha \Pi  \delta^{  Y(1) } \eta^{   Y(0) }  \theta^{ Y(1)Y(0) }$ with $\delta, \eta, \theta \geq 1$.
\end{enumerate} 
\end{corollary}

\section{Illustrations}
\label{sec::numerical-examples}

\subsection{Numerical Examples}

\citet{myers2011effects} simulated binary $(Z,U,A,Y)$ to investigate Z-Bias. They generated $(Z,U)$ according to $\pr(Z=1) = 0.5$ and $\pr(U=1) = \gamma_0$. The first set of their generative models is additive,
\begin{eqnarray}
\pr(A=1\mid U,Z) = \alpha_0+\alpha_1 U + \alpha_2 Z,\quad 
\pr(Y=1\mid U, A) = \beta_0 + \beta_1 U + \beta_2 A ,
\label{eq::additive}
\end{eqnarray}
where the coefficients are all positive. The second set of their generative models is multiplicative,
\begin{eqnarray}
\pr(A=1\mid U,Z) = \alpha_0 \alpha_1^U  \alpha_2^Z,\quad
\pr(Y=1\mid U, A) = \beta_0  \beta_1^U  \beta_2^A,
\label{eq::multi}
\end{eqnarray}
where the coefficients in \eqref{eq::additive} and \eqref{eq::multi} are all positive. They use simulation to show that Z-Bias arises under these models. In fact, in the above models, $Z$ and $U$ have monotonic effects on $A$ without additive or multiplicative interactions, and $U$ acts monotonically on $Y$, given $A$. Therefore, Corollaries \ref{thm::Zbias-binary-case} and \ref{thm::Zbias-binary-case-multi} imply that Z-Bias must occur. The qualitative conclusion follows immediately from our theory. However, our theory does not make statements about the magnitude of the bias, and for more details about the magnitude and finite sample properties, see \citet{myers2011effects}.

We further use three numerical examples to illustrate the role of the no-interaction assumptions required by Theorems \ref{thm::Zbias-scalar-case} and \ref{thm::Zbias-scalar-case-multi} and Corollaries \ref{thm::Zbias-binary-case} and \ref{thm::Zbias-binary-case-multi}. 
Recall the conditional probability of the treatment $A$, $p_{zu} = \pr(A=1\mid Z=z,U=u)$, and define the conditional probabilities of the outcome $Y$ as $r_{au} = \pr(Y=1\mid A=a, U=u)$, for $z,a,u=0,1.$ 
Table \ref{tb::examples} gives three examples, where monotonicity on the conditional distributions of $A$ and $Y$ hold, and there are both additive and multiplicative interactions. In all cases, the instrumental variable $Z$ is Bernoulli$(p=0.5)$, and the confounder $U$ is another independent Bernoulli$(\pi=0.5)$. In Case 1, the weaker condition \eqref{eq::weaker-conditions} holds, and our theory implies that Z-Bias arises. In Case 2, neither the condition in Theorem \ref{thm::Zbias-scalar-case-general} or \eqref{eq::weaker-conditions} holds, but Z-Bias still arises. Our conditions are only sufficient but not necessary. In Case 3, neither the condition in Theorem \ref{thm::Zbias-scalar-case-general} or \eqref{eq::weaker-conditions} holds, and Z-Bias does not arise.

\begin{table}[t]
\centering
\caption{Examples for the presence and absence of Z-Bias, in which $Z\sim$ Bernoulli$(0.5)$, $U\sim$ Bernoulli$(0.5)$, the conditional probability of the treatment $A$ is $p_{zu} =  \pr(A=1\mid Z=z,U=u) $, and the conditional probability of the outcome $Y$ is $r_{au} = \pr(Y=1\mid A=a, U=u)$. }\label{tb::examples}
\begin{tabular}{ccccccccccccc}
\hline
Case&$p_{11}$ & $p_{10}$ & $p_{01}$ & $p_{00}$ & $r_{11}$ & $r_{10} $ & $r_{01}$ & $r_{00}$ & $\ACE^\true$ & $\FACE$ & $\ACE^\adj$& Z-Bias \\
\hline
1&0.8 & 0.6& 0.2& 0.1& 0.08& 0.06& 0.02& 0.01& 0.0550 &0.0574& 0.0584 &YES\\
2&0.3 & 0.2&0.3&0.1&  0.03& 0.02& 0.03& 0.01&0.0050 &0.0076& 0.0077 &YES\\
3&0.5& 0.4&0.4& 0.1&0.04 & 0.04& 0.04& 0.01&0.0150 &0.0173& 0.0172&NO\\
\hline 
\end{tabular}
\end{table}

Finally, for binary $(Z,U,A,Y)$ we use Monte Carlo to compute the volume of the Z-Bias space, i.e., the parameter space of $p$, $\pi$, $p_{zu} $'s and $r_{au} $'s in which the adjusted estimator has higher bias than the unadjusted estimator. We randomly draw these ten probabilities from independent Uniform$(0,1)$ random variables, and for each draw of these probabilities we compute the average causal effect $\ACE^\true$, the unadjusted estimator $\FACE$ and the adjusted estimator $\ACE^\adj$. We plot the joint values of the biases $( \ACE^\adj - \ACE^\true,  \FACE - \ACE^\true )$ in Figure \ref{fg::numerical}.
The volume of the Z-Bias space can be approximated by the frequency that $\ACE^\adj$ deviates more from $\ACE^\true$ than $\FACE$. With $10^6$ random draws, our Monte Carlo gives an unbiased estimate for this volume as $0.6805$ with estimated standard error $0.0005$. Therefore, in about $68\%$ of the parameter space, the adjusted estimator is more biased than the unadjusted estimator.

\begin{figure}[!t]
\centering
\includegraphics[width = 0.7 \textwidth]{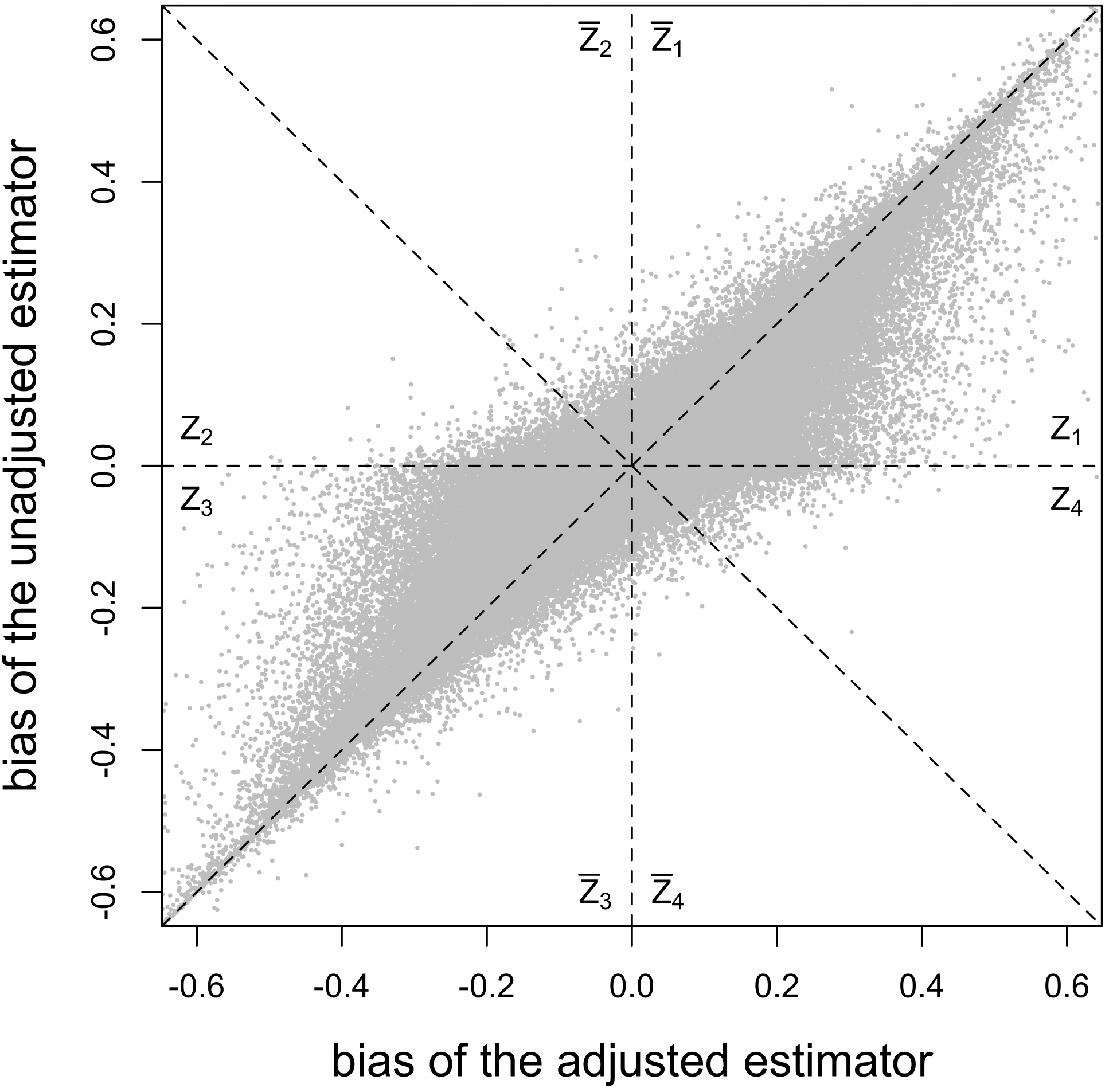}
\caption{Biases of the adjusted and unadjusted estimators over $10^6$ draws of the probabilities. In areas $(Z_1,Z_2,Z_3,Z_4)$ Z-Bias arises, and in areas $(\bar{Z}_1,\bar{Z}_2,\bar{Z}_3,\bar{Z}_4)$ Z-Bias does not arise.}\label{fg::numerical}
\end{figure}

\subsection{Real Data Examples}

\citet{bhattacharya2007instrumental} presented an example about the treatment effect of small classroom in the third grade on test scores for reading. Their instrumental variable analysis gave point estimate $8.73$ with standard error $2.01$. Without adjusting for the instrumental variable in the propensity score model, the point estimate was $6.00$ with estimated standard error $1.34$; adjusting for the instrumental variable, the point estimate was $2.97$ with estimated standard error $1.84$. The difference between the adjusted estimator and the instrumental variable estimator is larger than that between the unadjusted estimator and the instrumental variable estimator.

\citet[][Example 21.3]{wooldridge2010econometric} discusses estimating the effect of attaining at least seven years of education on fertility, with treatment $A$ being a binary indicator for at least seven years of education, outcome $Y$ being the number of living children, and instrumental variable $Z$ being a binary indicator if the woman was born in the first half of the year. Although the original data set of  \citet{wooldridge2010econometric} contains other variables, most of them are posttreatment variables, so we do not adjust for them in our analysis. 
The instrumental variable analysis gives point estimate $2.47$ with estimated standard error $0.59$. The unadjusted analysis gives point estimate $1.77$ with estimated standard error $0.07$. The adjusted analysis gives point estimate $1.76$ with estimated standard error $0.07$. Table \ref{tb::Real} summarizes the results. In this example, the adjusted and unadjusted estimators give similar results.


\begin{table}[t]
\centering
\caption{The example from \citet{wooldridge2010econometric}.}\label{tb::Real}
\begin{tabular}{ccccc}
  \hline
 & point estimate & standard error & lower confidence limit & upper confidence limit \\ 
  \hline
$\ACE^\true$ & 2.47 & 0.59 & 1.31 & 3.62 \\ 
$\FACE$ & 1.77 & 0.07 & 1.64 & 1.90 \\ 
$\ACE^\adj$ & 1.76 & 0.07 & 1.64 & 1.89 \\ 
   \hline
\end{tabular}
\end{table}

\section{Discussion}
\label{sec::discussion}

\subsection{Allowing for an Arrow from $Z$ to $Y$}

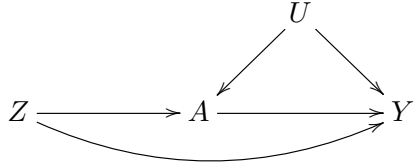
\begin{figure}[!t]
\centering
$$
\begin{xy}
\xymatrix{
& & & U \ar[dl] \ar[dr] \\
Z \ar[rr] \ar@/_1.5pc/[rrrr] & & A\ar[rr]  & & Y}
\end{xy}
$$
\caption{Directed Acyclic Graph for Z-Bias Allowing for an Arrow from $Z$ to $Y$}\label{fg::DAG-direct}
\end{figure}

When the variable $Z$ has an arrow to the outcome $Y$ as illustrated by Figure \ref{fg::DAG-direct}, the following generalization of Theorem \ref{thm::Zbias-scalar-case-general} holds.

\begin{theorem}
\label{thm::Zbias-scalar-case-general-direct}
Consider the causal diagram of Figure \ref{fg::DAG-direct} with scalar $Z$ and $U$, where $Z\ind U$ and $A\ind Y(a)\mid (Z,U)$ for $a=0$ and $1$. The result in \eqref{eq::result} holds if we replace Condition (a) of Theorem \ref{thm::Zbias-scalar-case-general} by
\begin{enumerate}
[(a')]
\item
$\pr(A=1\mid Z=z, U=u)$ and $\E(Y\mid A=a, Z=z, U=u)$ are non-decreasing in $z$ and $u$ for $a=0$ and $1$.
\end{enumerate}
\end{theorem}

However, when there is an arrow from $Z$ to $Y$, Theorem \ref{thm::Zbias-scalar-case-general-direct} is of little use in practice without strong substantive knowledge about the size of the direct effect of $Z$ on $Y$. In particular, neither Theorem \ref{thm::Zbias-scalar-case} nor Theorem \ref{thm::Zbias-scalar-case-multi} is true when an arrow from $Z$ to $Y$ is present. This reflects the fact that neither the absence of an additive nor the absence of a multiplicative interaction of $Z$ and $U$ on $A$ is sufficient to conclude that $E(Y\mid A=a,Z=z)$ is non-increasing in $z$ when $E(Y\mid A=a,U=u,Z=z)$ is non-decreasing in $z$ and $u$.

With a general instrumental variable and a general confounder, Theorem \ref{thm::general-iv-conf-weak} holds without any assumptions on the underlying causal diagram, and therefore it holds even if the variable $Z$ affects the outcome directly. 
However, Theorems \ref{thm::general-iv-conf} and \ref{thm::general-iv-conf-multi} no longer hold if an arrow from $Z$ to $Y$ is present as in Figure \ref{fg::DAG-direct}. This reflects the fact that the absence of an additive or multiplicative interaction of $U$ and $\Pi$ on $A$ no longer implies $\cov\{ \Pi ,  \nu_a(\Pi) \} \leq 0$ when $Z$ has a direct effect on $Y$, even if the remaining conditions of Theorems \ref{thm::general-iv-conf} and \ref{thm::general-iv-conf-multi} hold. Analogously, Theorems \ref{thm::general-iv-conf} and \ref{thm::general-iv-conf-multi} no longer hold if there exits an unmeasured common cause of $Z$ and $Y$ on the causal diagram in Figure \ref{fg::DAG}, even if $Z$ has no direct effect on $Y$.


\subsection{Extensions}

In \S\S \ref{sec::notation}--\ref{sec::general}, we discussed Z-Bias for the average causal effects. We can extend the results to distributional causal effects for general outcomes \citep{ju2010criteria} and causal risk ratios for binary or positive outcomes. Moreover, the results in \S\S \ref{sec::notation}--\ref{sec::general} are conditional on or within the strata of observed covariates. Similar results hold for causal effects averaged over observed covariates. We give more details in the Supplementary Material. In this paper we have given sufficient conditions for the presence of Z-Bias; future work could consider sufficient conditions for the absence of Z-Bias.

\subsection{Conclusion}

It is often suggested that we should adjust for all pretreatment covariates in observational studies. However, we show that in a wide class of models satisfying certain monotonicity, adjusting for an instrumental variable actually amplifies the impact of the unmeasured treatment-outcome confounding, which results in more bias than the unadjusted estimator.
In practice, we may not be sure about whether a covariate is a confounder, for which one needs to control, or perhaps instead an instrumental variable, for which control would only increase any existing bias due to unmeasured confounding. Therefore, a more practical approach, as suggested by \citet[][Chapter 18.2]{rosenbaum2010designobservational} and \citet{brookhart2010confounding}, may be to conduct analysis both with and without adjusting for the covariate. If two analyses give similar results, as in the example in Table \ref{tb::Real}, then we need not worry about Z-Bias; otherwise, we need additional information and analysis before making decisions.

%

\section*{Acknowledgments}
Peng Ding is partially supported by the U.S. Institute of Education Sciences, and
Tyler J. VanderWeele by the U.S. National Institutes of Health. 
The authors thank the Associate Editor and two reviewers for detailed and helpful comments.

\section*{Supplementary material}
\label{SM}

Supplementary Material available at \Bka~online includes all the proofs and extensions.

\bibliographystyle{biometrika}
\bibliography{Zbias.bib}

\begin{thebibliography}{49}
\expandafter\ifx\csname natexlab\endcsname\relax\def\natexlab#1{#1}\fi

\bibitem[{Ali et~al.(2014)Ali, Groenwold \& Klungel}]{ali2014propensity}
\textsc{Ali, M.~S.}, \textsc{Groenwold, R.~H.} \& \textsc{Klungel, O.~H.}
  (2014).
\newblock Propensity score methods and unobserved covariate imbalance:
  {C}omments on ``{S}queezing the balloon''.
\newblock \textit{Health Services Research} \textbf{49}, 1074--1082.

\bibitem[{Bhattacharya \& Vogt(2012)}]{bhattacharya2007instrumental}
\textsc{Bhattacharya, J.} \& \textsc{Vogt, W.~B.} (2012).
\newblock Do instrumental variables belong in propensity scores?
\newblock \textit{Int. J. Stat. Econ.} \textbf{9}, 107--127.

\bibitem[{Brookhart et~al.(2010)Brookhart, St{\"u}rmer, Glynn, Rassen \&
  Schneeweiss}]{brookhart2010confounding}
\textsc{Brookhart, M.~A.}, \textsc{St{\"u}rmer, T.}, \textsc{Glynn, R.~J.},
  \textsc{Rassen, J.} \& \textsc{Schneeweiss, S.} (2010).
\newblock Confounding control in healthcare database research: challenges and
  potential approaches.
\newblock \textit{Medical Care} \textbf{48}, S114--S120.

\bibitem[{Brooks \& Ohsfeldt(2013)}]{brooks2013squeezing}
\textsc{Brooks, J.~M.} \& \textsc{Ohsfeldt, R.~L.} (2013).
\newblock Squeezing the balloon: {P}ropensity scores and unmeasured covariate
  balance.
\newblock \textit{Health Services Research} \textbf{48}, 1487--1507.

\bibitem[{Chiba(2009)}]{chiba2009sign}
\textsc{Chiba, Y.} (2009).
\newblock The sign of the unmeasured confounding bias under various standard
  populations.
\newblock \textit{Biometrical Journal} \textbf{51}, 670--676.

\bibitem[{Cochran(1965)}]{cochran1965planning}
\textsc{Cochran, W.~G.} (1965).
\newblock The planning of observational studies of human populations (with
  discussion).
\newblock \textit{Journal of the Royal Statistical Society: Series A (General)}
  \textbf{128}, 234--266.

\bibitem[{Cox \& Wermuth(2003)}]{cox2003general}
\textsc{Cox, D.} \& \textsc{Wermuth, N.} (2003).
\newblock A general condition for avoiding effect reversal after
  marginalization.
\newblock \textit{Journal of the Royal Statistical Society: Series B
  (Statistical Methodology)} \textbf{65}, 937--941.

\bibitem[{d'Agostino(1998)}]{d1998tutorial}
\textsc{d'Agostino, R.~B.} (1998).
\newblock Tutorial in biostatistics: {P}ropensity score methods for bias
  reduction in the comparison of a treatment to a non-randomized control group.
\newblock \textit{Statistics in Medicine} \textbf{17}, 2265--2281.

\bibitem[{Ding \& Miratrix(2015)}]{ding2015adjust}
\textsc{Ding, P.} \& \textsc{Miratrix, L.~W.} (2015).
\newblock To adjust or not to adjust? {S}ensitivity analysis of {M}-{B}ias and
  {B}utterfly-{B}ias (with comments).
\newblock \textit{Journal of Causal Inference} \textbf{3}, 41--57.

\bibitem[{Ding \& VanderWeele(2016)}]{Ding::2015}
\textsc{Ding, P.} \& \textsc{VanderWeele, T.~J.} (2016).
\newblock Sensitivity analysis without assumptions.
\newblock \textit{Epidemiology} \textbf{27}, 368--377.

\bibitem[{Dorn(1953)}]{dorn1953philosophy}
\textsc{Dorn, H.~F.} (1953).
\newblock Philosophy of inferences from retrospective studies.
\newblock \textit{American Journal of Public Health and the Nations Health}
  \textbf{43}, 677--683.

\bibitem[{Esary et~al.(1967)Esary, Proschan \& Walkup}]{esary1967association}
\textsc{Esary, J.~D.}, \textsc{Proschan, F.} \& \textsc{Walkup, D.~W.} (1967).
\newblock Association of random variables, with applications.
\newblock \textit{The Annals of Mathematical Statistics} \textbf{38},
  1466--1474.

\bibitem[{Greenland(2003)}]{Greenland::2003}
\textsc{Greenland, S.} (2003).
\newblock Quantifying biases in causal models: {C}lassical confounding vs
  collider-stratification bias.
\newblock \textit{Epidemiology} \textbf{14}, 300--306.

\bibitem[{Greenland \& Robins(1986)}]{greenland1986identifiability}
\textsc{Greenland, S.} \& \textsc{Robins, J.~M.} (1986).
\newblock Identifiability, exchangeability, and epidemiological confounding.
\newblock \textit{International Journal of Epidemiology} \textbf{15}, 413--419.

\bibitem[{Heckman \& Navarro-Lozano(2004)}]{heckman2004using}
\textsc{Heckman, J.} \& \textsc{Navarro-Lozano, S.} (2004).
\newblock Using matching, instrumental variables, and control functions to
  estimate economic choice models.
\newblock \textit{Review of Economics and Statistics} \textbf{86}, 30--57.

\bibitem[{Hirano \& Imbens(2001)}]{hirano2001estimation}
\textsc{Hirano, K.} \& \textsc{Imbens, G.~W.} (2001).
\newblock Estimation of causal effects using propensity score weighting: An
  application to data on right heart catheterization.
\newblock \textit{Health Services and Outcomes Research Methodology}
  \textbf{2}, 259--278.

\bibitem[{Ju \& Geng(2010)}]{ju2010criteria}
\textsc{Ju, C.} \& \textsc{Geng, Z.} (2010).
\newblock Criteria for surrogate end points based on causal distributions.
\newblock \textit{Journal of the Royal Statistical Society: Series B
  (Statistical Methodology)} \textbf{72}, 129--142.

\bibitem[{Karlin \& Rinott(1980)}]{karlin1980classes}
\textsc{Karlin, S.} \& \textsc{Rinott, Y.} (1980).
\newblock Classes of orderings of measures and related correlation
  inequalities. {I}. {M}ultivariate totally positive distributions.
\newblock \textit{Journal of Multivariate Analysis} \textbf{10}, 467--498.

\bibitem[{Ma et~al.(2006)Ma, Xie \& Geng}]{ma2006collapsibility}
\textsc{Ma, Z.}, \textsc{Xie, X.} \& \textsc{Geng, Z.} (2006).
\newblock Collapsibility of distribution dependence.
\newblock \textit{Journal of the Royal Statistical Society: Series B
  (Statistical Methodology)} \textbf{68}, 127--133.

\bibitem[{Myers et~al.(2011)Myers, Rassen, Gagne, Huybrechts, Schneeweiss,
  Rothman, Joffe \& Glynn}]{myers2011effects}
\textsc{Myers, J.~A.}, \textsc{Rassen, J.~A.}, \textsc{Gagne, J.~J.},
  \textsc{Huybrechts, K.~F.}, \textsc{Schneeweiss, S.}, \textsc{Rothman,
  K.~J.}, \textsc{Joffe, M.~M.} \& \textsc{Glynn, R.~J.} (2011).
\newblock Effects of adjusting for instrumental variables on bias and precision
  of effect estimates.
\newblock \textit{American Journal of Epidemiology} \textbf{174}, 1213--1222.

\bibitem[{Neyman(1923[1990])}]{Neyman::1923}
\textsc{Neyman, J.} (1923[1990]).
\newblock On the application of probability theory to agricultural experiments.
  {E}ssay on principles. {S}ection 9. {T}ranslated by {D. M. D}abrowska and {T.
  P. S}peed.
\newblock \textit{Statistical Science} \textbf{5}, 465--472.

\bibitem[{Pearl(1995)}]{pearl1995causal}
\textsc{Pearl, J.} (1995).
\newblock Causal diagrams for empirical research (with discussion).
\newblock \textit{Biometrika} \textbf{82}, 669--688.

\bibitem[{Pearl(2000)}]{Pearl2000}
\textsc{Pearl, J.} (2000).
\newblock \textit{Causality: {M}odels, {R}easoning and {I}nference}.
\newblock Cambridge: Cambridge University Press.

\bibitem[{Pearl(2009)}]{Pearl2009a}
\textsc{Pearl, J.} (2009).
\newblock Letter to the editor.
\newblock \textit{Statistics in Medicine} \textbf{28}, 1415--1416.

\bibitem[{Pearl(2010)}]{pearl2010class}
\textsc{Pearl, J.} (2010).
\newblock On a class of bias-amplifying variables that endanger effect
  estimates.
\newblock In \textit{Proceedings of the Twenty-Sixth Conference on Uncertainty
  in Artificial Intelligence (UAI 2010)}, P.~Grunwald \& P.~Spirtes, eds.
  Association for Uncetainty in Artificial Intelligence, Corvallis, OR:
  425--432.

\bibitem[{Pearl(2011)}]{pearl2011invited}
\textsc{Pearl, J.} (2011).
\newblock Invited commentary: {U}nderstanding bias amplification.
\newblock \textit{American Journal of Epidemiology} \textbf{174}, 1223--1227.

\bibitem[{Pearl(2013)}]{pearl2013linear}
\textsc{Pearl, J.} (2013).
\newblock Linear models: {A} useful ``microscope'' for causal analysis.
\newblock \textit{Journal of Causal Inference} \textbf{1}, 155--170.

\bibitem[{Piegorsch et~al.(1994)Piegorsch, Weinberg \&
  Taylor}]{Piegorsch::1994}
\textsc{Piegorsch, W.~W.}, \textsc{Weinberg, C.~R.} \& \textsc{Taylor, J.~A.}
  (1994).
\newblock Non-hierarchical logistic models and case-only designs for assessing
  susceptibility in population-based case-control studies.
\newblock \textit{Statistics in Medicine} \textbf{13}, 153--162.

\bibitem[{Rosenbaum(2002)}]{rosenbaum2002observational}
\textsc{Rosenbaum, P.~R.} (2002).
\newblock \textit{Observational Studies}.
\newblock New York: Springer, 2nd ed.

\bibitem[{Rosenbaum(2010)}]{rosenbaum2010designobservational}
\textsc{Rosenbaum, P.~R.} (2010).
\newblock \textit{Design of Observational Studies}.
\newblock New York: Springer.

\bibitem[{Rosenbaum \& Rubin(1983)}]{Rosenbaum::1983}
\textsc{Rosenbaum, P.~R.} \& \textsc{Rubin, D.~B.} (1983).
\newblock The central role of the propensity score in observational studies for
  causal effects.
\newblock \textit{Biometrika} \textbf{70}, 41--55.

\bibitem[{Rothman et~al.(2008)Rothman, Greenland \& Lash}]{rothman2008modern}
\textsc{Rothman, K.~J.}, \textsc{Greenland, S.} \& \textsc{Lash, T.~L.} (2008).
\newblock \textit{Modern Epidemiology (3rd Edition)}.
\newblock Lippincott Williams \& Wilkins.

\bibitem[{Rubin(1974)}]{Rubin::1974}
\textsc{Rubin, D.~B.} (1974).
\newblock Estimating causal effects of treatments in randomized and
  nonrandomized studies.
\newblock \textit{Journal of Educational Psychology} \textbf{66}, 688--701.

\bibitem[{Rubin(2007)}]{rubin2007design}
\textsc{Rubin, D.~B.} (2007).
\newblock The design versus the analysis of observational studies for causal
  effects: {P}arallels with the design of randomized trials.
\newblock \textit{Statistics in Medicine} \textbf{26}, 20--36.

\bibitem[{Rubin(2008{\natexlab{a}})}]{rubin2008author}
\textsc{Rubin, D.~B.} (2008{\natexlab{a}}).
\newblock Author's reply.
\newblock \textit{Statistics in Medicine} \textbf{27}, 2741--2742.

\bibitem[{Rubin(2008{\natexlab{b}})}]{rubin2008objective}
\textsc{Rubin, D.~B.} (2008{\natexlab{b}}).
\newblock For objective causal inference, design trumps analysis.
\newblock \textit{The Annals of Applied Statistics} \textbf{2}, 808--840.

\bibitem[{Rubin(2009)}]{rubin2009should}
\textsc{Rubin, D.~B.} (2009).
\newblock Should observational studies be designed to allow lack of balance in
  covariate distributions across treatment groups?
\newblock \textit{Statistics in Medicine} \textbf{28}, 1420--1423.

\bibitem[{Shrier(2008)}]{Shrier2008}
\textsc{Shrier, I.} (2008).
\newblock Letter to the editior.
\newblock \textit{Statistics in Medicine} \textbf{27}, 2740--2741.

\bibitem[{Shrier(2009)}]{Shrier2009}
\textsc{Shrier, I.} (2009).
\newblock Propensity scores.
\newblock \textit{Statistics in Medicine} \textbf{28}, 1315--1318.

\bibitem[{Sj{\"o}lander(2009)}]{Sjolander2009}
\textsc{Sj{\"o}lander, A.} (2009).
\newblock Propensity scores and {M}-structures.
\newblock \textit{Statistics in Medicine} \textbf{28}, 1416--1420.

\bibitem[{Spirtes et~al.(2000)Spirtes, Glymour \&
  Scheines}]{spirtes2000causation}
\textsc{Spirtes, P.}, \textsc{Glymour, C.~N.} \& \textsc{Scheines, R.} (2000).
\newblock \textit{Causation, Prediction, and Search}.
\newblock Cambridge: MIT press, 2nd ed.

\bibitem[{VanderWeele(2008)}]{vanderweele2008sign}
\textsc{VanderWeele, T.~J.} (2008).
\newblock The sign of the bias of unmeasured confounding.
\newblock \textit{Biometrics} \textbf{64}, 702--706.

\bibitem[{VanderWeele \& Robins(2010)}]{vanderweele2010signed}
\textsc{VanderWeele, T.~J.} \& \textsc{Robins, J.~M.} (2010).
\newblock Signed directed acyclic graphs for causal inference.
\newblock \textit{Journal of the Royal Statistical Society: Series B
  (Statistical Methodology)} \textbf{72}, 111--127.

\bibitem[{VanderWeele \& Shpitser(2011)}]{vanderweele2011new}
\textsc{VanderWeele, T.~J.} \& \textsc{Shpitser, I.} (2011).
\newblock A new criterion for confounder selection.
\newblock \textit{Biometrics} \textbf{67}, 1406--1413.

\bibitem[{Walker(2013)}]{walker2013matching}
\textsc{Walker, A.~M.} (2013).
\newblock Matching on provider is risky.
\newblock \textit{Journal of Clinical Epidemiology} \textbf{66}, S65--S68.

\bibitem[{Wooldridge(2016)}]{wooldridge2009should}
\textsc{Wooldridge, J.} (2016).
\newblock Should instrumental variables be used as matching variables?
\newblock \textit{Research in Economics} \textbf{70}, 232--237.

\bibitem[{Wooldridge(2010)}]{wooldridge2010econometric}
\textsc{Wooldridge, J.~M.} (2010).
\newblock \textit{Econometric Analysis of Cross Section and Panel Data}.
\newblock Cambridge: MIT Press, 2nd ed.

\bibitem[{Xie et~al.(2008)Xie, Ma \& Geng}]{xie2008some}
\textsc{Xie, X.}, \textsc{Ma, Z.} \& \textsc{Geng, Z.} (2008).
\newblock Some association measures and their collapsibility.
\newblock \textit{Statistica Sinica} \textbf{18}, 1165--1183.

\bibitem[{Yang et~al.(1999)Yang, Khoury, Sun \& Flanders}]{Yang::1999}
\textsc{Yang, Q.}, \textsc{Khoury, M.~J.}, \textsc{Sun, F.} \&
  \textsc{Flanders, W.~D.} (1999).
\newblock Case-only design to measure gene-gene interaction.
\newblock \textit{Epidemiology} \textbf{10}, 167--170.

\end{thebibliography}

\newpage  
\begin{center}
\bf \Large 
Supplementary material for ``Instrumental variables as bias amplifiers with general outcome and confounding''
\end{center}

\bigskip 

\renewcommand {\theequation}{S\arabic{equation}}
\renewcommand {\thelemma}{S\arabic{lemma}}
\renewcommand {\thecorollary}{S\arabic{corollary}}
\renewcommand {\thefigure}{S\arabic{figure}}
\renewcommand {\thetheorem}{S\arabic{theorem}}

\section*{Appendix 1.\quad  Lemmas and Their Proofs}

In order to prove the main results, we need to invoke the following lemmas. Some of them are from the literature, and some of them are new and of independent interest.

Lemma \ref{lemma::epw1967} is from \citet[][Theorem 2.1]{esary1967association}.

\begin{lemma}
\label{lemma::epw1967}
Let $f(\cdot)$ and $g(\cdot)$ be functions with $K$ real-valued arguments, which are both non-decreasing in each of their arguments. If $U = (U_1, \ldots, U_K)$ is a multivariate random variable with $K$ mutually independent components, then $\cov\{  f(U), g(U) \} \geq 0.$
\end{lemma}

Lemma \ref{lemma::vw2008} is from \citet{vanderweele2008sign}, and
Lemmas \ref{lemma::chiba2009-1} and \ref{lemma::chiba2009-2} are from \citet{chiba2009sign}. 

\begin{lemma}
\label{lemma::vw2008}
For a univariate $U$ or a multivariate $U$ with mutually independent components, if for $a=1$ and $0$,
$Y(a) \ind A\mid U$,
$\E(Y\mid A=a, U=u)$ is non-decreasing in each component of $u$, and
$\pr(A=1\mid U=u)$ is non-decreasing in each component of $u,$
then
$
\E(Y\mid A=1) \geq \E\{ Y (1 ) \}  
$ 
and
$
\E(Y\mid A=0) \leq \E\{ Y(0) \} .
$
\end{lemma}

\begin{lemma}
\label{lemma::chiba2009-1}
For a univariate $U$ and a multivariate $U$ with mutually independent components, if
$Y(0) \ind A\mid U$,
$\E(Y\mid A=0, U=u)$ is non-decreasing in each component of $u$, and
$\pr(A=1\mid U=u)$ is non-decreasing in each component of $u,$
then
$
\E(Y\mid A=0) \leq \E\{ Y(0) \mid A=1 \} .
$
\end{lemma}

\begin{lemma}
\label{lemma::chiba2009-2}
For a univariate $U$ and a multivariate $U$ with mutually independent components, if
$Y (1)  \ind A\mid U$, 
$\E(Y\mid A=1, U=u)$ is non-decreasing in each component of $u$, and
$\pr(A=1\mid U=u)$ is non-decreasing in each component of $u,$
then
$
\E(Y\mid A=1) \geq \E\{ Y(1) \mid A=0 \} .
$
\end{lemma}

Lemma \ref{lemma::interactions}, extending \citet{rothman2008modern}, states that under monotonicity, no additive interaction implies non-positive multiplicative interactions for both presence and absence of the outcome.

\begin{lemma}
\label{lemma::interactions}
If $p_{11} \geq \max(p_{10}, p_{01})$, $\min(p_{10}, p_{01}) \geq p_{00} >0$, and $p_{11}-p_{10}-p_{01}+p_{00} = 0$, then
\begin{eqnarray}\label{eq::weaker-conditions-supp}
\frac{p_{11}p_{00}}{p_{10}p_{01}} \leq 1, \quad 
\frac{ (1-p_{11}) (1-p_{00}) }{ (1-p_{10})(1-p_{01}) } \leq 1.
\end{eqnarray} 
\end{lemma}

\begin{proof}
[of Lemma \ref{lemma::interactions}]
Define $\RR_{11} = p_{11}/p_{00} \geq 1$, $\RR_{10} = p_{10} / p_{00} \geq 1$ and $\RR_{01} = p_{01} / p_{00} \geq 1$. Then 
$p_{11}-p_{10}-p_{01}+p_{00} = 0$ implies
$
\RR_{11} = \RR_{10} + \RR_{01} - 1,
$
which further implies
\begin{eqnarray*}
\frac{p_{11}p_{00}}{p_{10}p_{01}}  
&=& \frac{\RR_{11}}{ \RR_{10} \RR_{01}}  = 1 + \frac{1}{ \RR_{10} \RR_{01}  }   (   \RR_{11}  - \RR_{10} \RR_{01}) \\
&=& 1 + \frac{1}{ \RR_{10} \RR_{01}  }   (   \RR_{10} + \RR_{01} - 1 - \RR_{10} \RR_{01}) \\
&=& 1 - \frac{1}{ \RR_{10} \RR_{01}  }   (   \RR_{10} -1) ( \RR_{01} - 1) \leq 1.
\end{eqnarray*}

The second inequality of \eqref{eq::weaker-conditions-supp} follows from
\begin{eqnarray*}
&&\frac{ (1-p_{11}) (1-p_{00}) }{ (1-p_{10})(1-p_{01}) } 
= 
1+  \frac{ (1-p_{11}) (1-p_{00}) - (1-p_{10})(1-p_{01})  }{ (1-p_{10})(1-p_{01}) } \\
&=& 1 + \frac{1}{(1-p_{10})(1-p_{01}) } \left\{  (   1-p_{11}-p_{00}+p_{11}p_{00}       )
- (1-p_{10} - p_{01} + p_{10}p_{01})
\right\} \\
&=&1 + \frac{1}{(1-p_{10})(1-p_{01}) } (p_{11}p_{00} - p_{10}p_{01}) \\
&=& 1 + \frac{p_{10}p_{01}}{(1-p_{10})(1-p_{01}) } \left(  \frac{p_{11}p_{00}}{p_{10}p_{01}}  - 1 \right) \leq 1.
\end{eqnarray*}
\end{proof}

Lemma \ref{lemma::interactions} is about interaction between two binary causes, and for our discussion we need to extend it to interaction between two general causes. Lemma \ref{lemma::general-interaction} extends \citet{Piegorsch::1994} and \citet{Yang::1999} by relating the conditional association between two independent causes given the outcome to the interaction between the two causes on the outcome.

\begin{lemma}
\label{lemma::general-interaction}
If $Z\ind U$, and $\pr(A=1\mid Z=z, U=u) =  \beta(z) +  \gamma(u)$ with $\beta(z)$ and $\gamma(u)$ non-decreasing in $z$ and $u$, then for both $a=1$ and $0$ and for all values of $u$ and $z$,
$$
\frac{\partial F( u\mid A=a, Z=z)}{\partial z}  \geq  0 ,
$$
i.e., $U$ has non-positive distributional dependence on $Z$, given $A$.
\end{lemma}

\begin{proof}
[of Lemma \ref{lemma::general-interaction}]
For a fixed $u$ and $z_1 > z_0$, we define
\begin{eqnarray*}
p_{11} &=& \pr(A=1\mid U>u, Z=z_1) = \int_u^\infty \{  \beta(z_1) +  \gamma(u')\} F(\d u')/\{ 1-F(u) \}   ,\\
p_{10} &=& \pr(A=1\mid U>u, Z=z_0) = \int_u^\infty \{  \beta(z_0) +  \gamma(u')\} F(\d u')/\{ 1-F(u) \}   ,\\
p_{01} &=& \pr(A=1\mid U\leq u, Z=z_1) = \int^u_{-\infty}  \{  \beta(z_1) +  \gamma(u')\} F(\d u')/ F(u)    ,\\
p_{00} &=& \pr(A=1\mid U\leq u, Z=z_0) = \int^u_{-\infty}  \{  \beta(z_0) +  \gamma(u')\} F(\d u')/ F(u)  ,
\end{eqnarray*}
following from the additive model of $A$ and $Z\ind U.$

Because $\beta(z_1)\geq \beta(z_0),$ it is straightforward to show that $p_{11} \geq  p_{10}$ and $p_{01}\geq p_{00}$. Because $\gamma(u)$ is increasing in $u$, we have
$$
p_{11} \geq \beta(z_1) + \gamma(u),\quad
p_{10} \geq \beta(z_0) + \gamma(u),\quad
p_{01} \leq  \beta(z_1) + \gamma(u),\quad
p_{00} \leq  \beta(z_0) + \gamma(u),
$$
which imply
$
p_{11}  \geq  p_{01} 
$
and
$
p_{10} \geq p_{00}. 
$
We further have
\begin{eqnarray*}
&&p_{11} - p_{10} - p_{01} + p_{00} \\
&=& \int_u^\infty \{  \beta(z_1)  - \beta(z_0 )\} F(\d u')/\{ 1-F(u) \} 
- \int^u_{-\infty}  \{  \beta(z_1)  - \beta(z_0)  \} F(\d u')/ F(u)  \\
&=& 0.
\end{eqnarray*}
The four probabilities $(p_{11}, p_{10}, p_{01}, p_{00})$ satisfy the conditions in Lemma \ref{lemma::interactions}, Therefore, \eqref{eq::weaker-conditions} holds.
Replacing the probabilities in \eqref{eq::weaker-conditions} by their definitions above, we have
\begin{eqnarray*}
&&\frac{   \pr(A=1\mid U>u, Z=z_1)  \pr(A=1\mid U\leq u, Z=z_0)   }
{    \pr(A=1\mid U>u, Z=z_0)    \pr(A=1\mid U\leq u, Z=z_1)    }  \leq  1 \\
&\Longleftrightarrow&
\frac{\pr(A=1\mid U>u, z_1)}{ \pr(A=1\mid U\leq u, z_1)  } 
\leq  \frac{\pr(A=1\mid U>u, z_0)}{ \pr(A=1\mid U\leq u, z_0)  } ,
\end{eqnarray*}
and
\begin{eqnarray*}
&&\frac{    \pr(A=0\mid U>u, Z=z_1)   \pr(A=0\mid U\leq u, Z=z_0)    }
{    \pr(A=0\mid U>u, Z=z_0)    \pr(A=0\mid U\leq u, Z=z_1)    } \leq 1 \\
&\Longleftrightarrow&
\frac{\pr(A=0 \mid U>u, z_1 )}{ \pr(A=0 \mid U\leq u, z_1 )  } 
\leq   \frac{\pr(A=0 \mid U>u, z_0 )}{ \pr(A=0 \mid U\leq u, z_0 )  } . 
\end{eqnarray*}
%
%
%
%
%
%
%
%
%
%
%
Therefore, for both $a=1$ and $0$ and for all values of $u$,
\begin{eqnarray}
\label{eq::distributional-interaction}
\frac{\pr(A=a\mid U>u, Z=z)}{ \pr(A=a\mid U\leq u, Z=z)  }
\end{eqnarray}
is non-increasing in $z$.
Because of the independence of $Z$ and $U$, we have
\begin{eqnarray*}
&&F(u\mid A=a, Z=z) \\
&=& \frac{ \pr(U\leq u,  A=a\mid  Z=z) }{   \pr(  A=a\mid  Z=z)   }\\
&=& \frac{ \pr(U\leq u) \pr  ( A=a\mid U\leq u,  Z=z) }{   \pr(U\leq u) \pr  ( A=a\mid U\leq u,  Z=z)
+   \pr(U >  u) \pr  ( A=a\mid U> u,  Z=z)  }\\
&=&  \left\{    1 + \frac{  \pr(U >  u)   }{ \pr(U\leq u) }  \times 
\frac{ \pr  ( A=a\mid U> u,  Z=z)   }{ \pr  ( A=a\mid U\leq u,  Z=z)  }       \right\}^{-1}.
\end{eqnarray*}
Therefore, $F( u\mid A=a, Z=z)$ is a non-increasing function of \eqref{eq::distributional-interaction}, and the conclusion holds. 
\end{proof}

Lemmas \ref{lemma::interactions} and \ref{lemma::general-interaction} above hold under the assumption of no additive interaction, and the following two lemmas state similar results under the assumption of no multiplicative interaction.

\begin{lemma}
\label{lemma::inter-multi}
If $p_{11}\geq \max(p_{10}, p_{01}), \min(p_{10}, p_{01} ) \geq p_{00}$, and $p_{11}p_{00} = p_{10}p_{01}$, then 
$$
p_{11} - p_{10} - p_{01} + p_{00} \geq 0,\quad 
\frac{(1-p_{11}) (1-p_{00}) }{  (1-p_{10}) (1-p_{01}) } \leq 1.
$$
\end{lemma}

\begin{proof}
[of Lemma \ref{lemma::inter-multi}]
Using the same notation in the proof of Lemma \ref{lemma::interactions}, $p_{11}p_{00} = p_{10}p_{01}$ implies $\RR_{11} = \RR_{10}\RR_{01}$, with
$\RR_{10} \geq 1,\RR_{01} \geq 1$, and $ \RR_{11}\geq 1.$ Therefore, 
\begin{eqnarray*}
p_{11} - p_{10} - p_{01} + p_{00} 
= p_{00}(\RR_{10}\RR_{01}- \RR_{10} - \RR_{01} + 1) 
=p_{00} (\RR_{10} - 1) (\RR_{01} - 1) \geq 0,
\end{eqnarray*}
which further implies that
\begin{eqnarray*}
\frac{(1-p_{11}) (1-p_{00}) }{  (1-p_{10}) (1-p_{01}) } 
&=&1  + \frac{1}{  (1-p_{10}) (1-p_{01})  }  \left\{  
(1-p_{11}) (1-p_{00}) -   (1-p_{10}) (1-p_{01})
\right\} \\
&=& 1  -  \frac{   p_{11} - p_{10} - p_{01} + p_{00}   }{  (1-p_{10}) (1-p_{01})  }  \leq 1.
\end{eqnarray*}
\end{proof}

\begin{lemma}
\label{lemma::inter-multi-general}
If $Z\ind U$, and $\pr(A=1\mid Z=z, U=u) =  \beta(z) \gamma(u)$ with $\beta(z)>0$ and $\gamma(u)>0$ non-decreasing in $z$ and $u$, then $Z\ind U\mid A=1$, and for all values of $u$ and $z$,
$$
\frac{\partial F( u\mid A=0, Z=z)}{\partial z}  \geq  0   ,
$$
i.e., $U$ has non-positive distributional dependence on $Z$, given $A=0$.
\end{lemma}

\begin{proof}
[of Lemma \ref{lemma::inter-multi-general}]
For a fixed $u$ and $z_1>z_0$, we define
\begin{eqnarray*}
p_{11} &=& \pr(A=1\mid U>u, Z=z_1) = \beta(z_1)   \int_u^\infty   \gamma(u')  F(\d u')/\{ 1-F(u) \}   ,\\
p_{10} &=& \pr(A=1\mid U>u, Z=z_0) = \beta(z_0)   \int_u^\infty     \gamma(u')  F(\d u')/\{ 1-F(u) \}   ,\\
p_{01} &=& \pr(A=1\mid U\leq u, Z=z_1) =  \beta(z_1)  \int^u_{-\infty}   \gamma(u')  F(\d u')/ F(u)    ,\\
p_{00} &=& \pr(A=1\mid U\leq u, Z=z_0) =  \beta(z_0)  \int^u_{-\infty} \gamma(u')  F(\d u')/ F(u)  ,
\end{eqnarray*}
following from the multiplicative model of $A$ and $Z\ind U.$
Because $\beta(z_1)\geq \beta(z_0),$ we have $p_{11} \geq  p_{10}$ and $p_{01}\geq p_{00}$. Because $\gamma(u)$ is increasing in $u$, we have
$$
p_{11} \geq \beta(z_1)  \gamma(u),\quad
p_{10} \geq \beta(z_0)  \gamma(u),\quad
p_{01} \leq  \beta(z_1)  \gamma(u),\quad
p_{00} \leq  \beta(z_0)  \gamma(u),
$$
which imply
$
p_{11}  \geq  p_{01} 
$
and
$
p_{10} \geq p_{00}. 
$
We can further verify $ ( p_{11} p_{00} ) /  ( p_{10}  p_{01} ) = 1 .$
Because the four probabilities $(p_{11}, p_{10}, p_{01}, p_{00})$ satisfy the conditions in Lemma \ref{lemma::inter-multi}, we have
$
\{ (1-p_{11}) (1-p_{00}) \} / \{ (1-p_{10})(1-p_{01}) \} \leq 1.
$
Replacing the probabilities by their definitions, we have
\begin{eqnarray*}
\frac{   \pr(A=1\mid U>u, Z=z_1)  \pr(A=1\mid U\leq u, Z=z_0)   }
{    \pr(A=1\mid U>u, Z=z_0)    \pr(A=1\mid U\leq u, Z=z_1)    }  &=& 1, \\
\frac{    \pr(A=0\mid U>u, Z=z_1)   \pr(A=0\mid U\leq u, Z=z_0)    }
{    \pr(A=0\mid U>u, Z=z_0)    \pr(A=0\mid U\leq u, Z=z_1)    } &\leq& 1 . 
\end{eqnarray*}
Following the same logic of the proof of Lemma \ref{lemma::general-interaction}, we can prove that $Z\ind U\mid A=1$, and $Z$ has non-positive distributional association on $U$, given $A=0$.
\end{proof}

Define $f=\pr(A=1)$ to be the proportion of the population under treatment. 
The average causal effect for the whole population can be written as a convex combination of the average causal effects for the treated and control populations: 
$$
\ACE^\true = \E\{  Y(1) \} - \E\{  Y(0) \} =   f \ACE_1^\true + (1-f) \ACE_0^\true.
$$
Analogously, with a scalar instrumental variable, the adjusted estimator for the whole population can be written as
$$
\ACE^\adj = \int  \mu_1(z) F(\d z)  -  \int   \mu_0(z) F(\d z)  
= f \ACE_1^\adj  + (1-f) \ACE_0 ^\adj ,
$$
and with a general instrumental variable, 
$$
\ACE^\adj = \int  \nu_1(\pi) F(\d \pi)  -  \int   \nu_0(\pi) F(\d \pi)  
= f \ACE_1^\adj  + (1-f) \ACE_0^\adj .
$$

\begin{lemma}
\label{lemma::adjust-formula}
With a scalar instrumental variable $Z$, the differences between the adjusted and unadjusted estimators are 
\begin{eqnarray*}
\ACE_1^\adj - \FACE &=&  -  \frac{  \cov\{   \Pi(Z) , \mu_0(Z)  \}    }{   f(1-f)  } ,\\
\ACE_0^\adj - \FACE &=&  -  \frac{  \cov\{   \Pi(Z) , \mu_1(Z)  \}    }{   f(1-f)  } ,\\
\ACE^\adj -  \FACE &=& -  \frac{  \cov\{   \Pi(Z) , \mu_0(Z)  \}    }{  1-f } -  \frac{  \cov\{   \Pi(Z) , \mu_1(Z)  \}    }{   f  }  .
\end{eqnarray*}
With a general instrumental variable $Z$, the above formulas hold if we replace $\Pi(Z)$ by $\Pi$ and $\mu_a(Z) = \E(Y\mid A=a,Z)$ by $\nu_a(\Pi) = \E(Y\mid A=a, \Pi).$
\end{lemma}

\begin{proof}
[of Lemma \ref{lemma::adjust-formula}]
The difference $\ACE_1^\adj - \FACE $ is equal to
\begin{eqnarray*}
&&\ACE_1^\adj - \FACE  \\
&=& \E(Y\mid A=0) -  \int   \mu_0(z) F(\d z\mid A=1)  \\
&=&  \int   \mu_0(z) F(\d z\mid A=0) -  \int   \mu_0(z) F(\d z\mid A=1) \\
&=& \frac{    \int   \mu_0(z) \{ 1- \Pi(z) \}  F(\d z)    }{1-f}
- \frac{   \int    \mu_0(z)  \Pi(z)  F(\d z)    }{f} \\
&=& \frac{1}{f(1-f)}  \Big [  
\E\{     \mu_0(Z) (1- \Pi(Z) )   \}  \E\{   \Pi(Z)   \} 
-  \E\{   \mu_0(Z)  \Pi(Z) \} \E\{ 1 - \Pi(Z)  \} 
\Big ] \\
&=& \frac{1}{f(1-f)} \Big [     \E\{  \mu_0(Z)  \} \E\{  \Pi(Z)   \}  -   \E\{   \mu_0(Z)  \Pi(Z) \}  \Big ] \\
&=& -\frac{  \cov\{   \Pi(Z) , \mu_0(Z)  \}    }{   f(1-f)  }.
\end{eqnarray*}

Similarly, the difference $\ACE_0^\adj - \FACE$ is equal to
\begin{eqnarray*}
\ACE_0^\adj - \FACE 
&=& \int  \mu_1(z) F(\d z \mid A=0)  - \int  \mu_1(z) F(\d z \mid A=1) \\
&=&  -\frac{  \cov\{   \Pi(Z) , \mu_1(Z)  \}    }{   f(1-f)  }.
\end{eqnarray*}

Therefore, the difference $\ACE^\adj -  \FACE$ is equal to
\begin{eqnarray*}
\ACE^\adj -  \FACE 
&=&f (\ACE_1^\adj - \FACE ) + (1-f) (\ACE_0^\adj - \FACE) \\
&=& -  \frac{  \cov\{   \Pi(Z) , \mu_0(Z)  \}    }{  1-f } -  \frac{  \cov\{   \Pi(Z) , \mu_1(Z)  \}    }{   f  }  .
\end{eqnarray*}

Analogously, we can prove the results for general instrumental variables.
\end{proof}

%
%
%

\section*{Appendix 2. \quad Proofs of Theorems and Corollaries in the Main Text}


\begin{proof}
[of Theorem \ref{thm::Zbias-scalar-case-general}]
Because
$
\Pi(z) =  \pr(A=1\mid Z=z)  
$
and
$
\pr(A=1\mid U=u) 
$
are non-decreasing in $z$ and $u$, and $  \E(Y\mid A=a, U=u)$ is non-decreasing in $u$ for both $a=0$ and $1$, 
the unadjusted estimator, $\FACE$, is larger than or equal to $\ACE^\true, \ACE_1^\true$ and $\ACE_0^\true$, according to Lemmas \ref{lemma::vw2008}--\ref{lemma::chiba2009-2}.

Because $\Pi(Z)$ is non-decreasing and $\mu_a(Z)$ is non-increasing in $Z$ for both $a=0$ and $1$, their covariance is non-positive according to Lemma \ref{lemma::epw1967}, i.e.,
$
\cov\{ \Pi(Z),  \mu_a(Z)  \} \leq 0.
$

Because the differences between all the adjusted estimators, $\ACE_1^\adj$, $\ACE_0^\adj$ and $\ACE^\adj$, and the unadjusted estimator, $\FACE$, are negative constants multiplied by $\cov\{ \Pi(Z),  \mu_a(Z)  \}$, according to Lemma \ref{lemma::adjust-formula} all of 
$
\ACE_1^\adj,
$
$
\ACE_0^\adj,
$ 
and
$ 
\ACE^\adj 
$
are larger or equal to $\FACE.$
\end{proof}

\begin{proof}
[of Theorem \ref{thm::Zbias-scalar-case}]
The independence of $Z$ and $U$ implies that
\begin{eqnarray*}
\pr(A=1\mid Z=z)  &=& \int \pr(A=1\mid Z=z, U=u) F( \d u) = \beta(z) + \E\{ \gamma(U) \}  ,\\
\pr(A=1\mid U=u) &=& \int \pr(A=1\mid Z=z, U=u) F(\d z) = \E \{ \beta(Z)  \} + \gamma(u)
\end{eqnarray*}
are non-decreasing in $z$ and $u.$ Therefore, according to Theorem \ref{thm::Zbias-scalar-case-general} we need only to verify that $\E(Y\mid A=a, Z= z)$ in non-increasing in $z$ for both $a=0$ and $1.$

Because $Z\ind U$ and $\pr(A=1\mid Z=z, U=u) =  \beta(z) +  \gamma(u)$ with non-decreasing $\beta(z)$ and $\gamma(u)$, we can apply Lemma \ref{lemma::general-interaction}, and conclude that 
$
 \partial F(u\mid A=a, Z=z) / \partial z   \geq 0.
$

Write the essential infimum and supremum of $U$ given $(A=a, Z=z)$ as $\underline{u}(a,z)$ and $\overline{u}(a)$, with the later depending only on $a$ according to Condition (c) of Theorem \ref{thm::Zbias-scalar-case}. Because $Y\ind Z\mid (A,U)$, integration or summation by parts gives
\begin{eqnarray*}
&&\E(Y\mid A=a, Z=z) \\
&=&\int \E(Y\mid A=a, Z=z, U=u) F(\d u \mid A=a, Z=z) \\
&=& \int m_a(u) F(\d u \mid A=a, Z=z) \\
&=& m_a(u) F(u  \mid A=a,Z=z) |_{u=\underline{u}(a,z)}^{u=\overline{u}(a)}
- \int   \left\{   \frac{\partial  m_a(u) }{ \partial u} \right\}  F(u\mid A=a, Z=z) \d u\\
&=& m_a\{ \overline{u}(a)   \} -  \int  \left\{   \frac{\partial  m_a(u) }{ \partial u} \right\} F(u\mid A=a, Z=z) \d u.
\end{eqnarray*}
Therefore, its derivative with respect to $z$,
\begin{eqnarray*}
\frac{\partial \E(Y\mid A=a, z)    }{\partial z} 
&=&  -
\frac{\partial }{\partial z} \int  \left\{   \frac{\partial  m_a(u) }{ \partial u} \right\} F(u\mid A=a, Z=z) \d u  \\
&=& 
-\int  \left\{   \frac{\partial  m_a(u) }{ \partial u} \right\}   \left\{    \frac{\partial F(u\mid A=a,Z=z) }{\partial z}    \right\}  \d u ,
\end{eqnarray*}
is smaller than or equal to zero, because $\partial m_a(u) /\partial u \geq 0$ for both $a=0$ and $1$ and for all $u$.
\end{proof}

\begin{proof}
[of Corollary \ref{thm::Zbias-binary-case}]
According to Theorem \ref{thm::Zbias-scalar-case-general} we need only to verify that $\mu_a(z) =  \E(Y\mid A=a, Z=z)$ is non-increasing in $z$ for both $a=0$ and $1.$
Following Lemma \ref{lemma::interactions}, for binary and independent $Z$ and $U,$ monotonicity and no additive interaction imply \eqref{eq::weaker-conditions-supp},
which, according to Bayes' Theorem, is equivalent to
\begin{eqnarray}
\frac{  \pr(A=1\mid Z=1, U=1)    \pr(A=1\mid Z=0, U=0)  }
{   \pr(A=1\mid Z=1, U=0)    \pr(A=1\mid Z=0, U=1)   } = \OR_{ZU|A=1} \leq 1, \label{eq::or1} \\
\frac{  \pr(A=0\mid Z=1, U=1)    \pr(A=0\mid Z=0, U=0)  }
{   \pr(A=0\mid Z=1, U=0)    \pr(A=0\mid Z=0, U=1)   } = \OR_{ZU|A=0}  \leq 1  . \label{eq::or0}
\end{eqnarray}
The above inequalities \eqref{eq::or1} and \eqref{eq::or0} state that $Z$ and $U$ have negative association given each level of $A$, and therefore $\pr(U=1\mid A=a, Z=z)$ is non-increasing in $z$ for both $a=1$ and $0.$

Because $m_a(1) \geq  m_a(0)$ and 
\begin{eqnarray*}
\mu_a(z) &=& \E(Y\mid A=a, Z=z) \\
&=& \sum_{u=0,1}  \E(Y\mid A=a, Z=z, U=u) \pr(U=u\mid A=a,  Z=z) \\
&=& m_a(1) \pr(U=1\mid A=a, Z=z) + m_a(0)  \{ 1 - \pr(U=1\mid A=a, Z=z) \} \\
&=& \{ m_a(1) - m_a(0) \}  \pr(U=1\mid A=a, Z=z)
+ m_a(0),
\end{eqnarray*}
we know that
$\mu_a(z)$ is non-decreasing in $\pr(U=1\mid A=a, Z=z)$. Therefore, $\mu_a(z)$ is non-increasing in $z$ for both $a=1$ and $0.$  
\end{proof}

\begin{proof}
[of Theorem \ref{thm::Zbias-scalar-case-multi}]
Because of the independence of $Z$ and $U$, we have $\pr(A=1\mid Z=z) = \beta(z)E\{  \gamma(U) \}$ and $\pr(A=1\mid U=u) = E\{ \beta(Z)  \}  \gamma(u)$ are non-decreasing in $z$ and $u.$ According to Lemma \ref{lemma::inter-multi-general}, the multiplicative model of $A$ also implies that for both $a=1$ and $0$ and for all $z$ and $u$,
$
 \partial F(u\mid A=a, Z=z) / \partial z   \geq 0.
$
Following exactly the same steps of the proof of Theorem \ref{thm::Zbias-scalar-case}, we can prove Theorem \ref{thm::Zbias-scalar-case-multi}.
\end{proof}

\begin{proof}
[of Corollary \ref{thm::Zbias-binary-case-multi}]
For binary and independent $Z$ and $U$, monotonicity, no multiplicative interaction, and Lemma \ref{lemma::inter-multi} imply
\begin{eqnarray}
\label{eq::multi-binary}
\frac{p_{11}p_{00}}{p_{10}p_{01}} = 1\leq 1,\quad
\frac{(1-p_{11}) (1-p_{00}) }{ (1-p_{10}) (1-p_{01})  } \leq 1.
\end{eqnarray}
With the above results in \eqref{eq::multi-binary}, the rest of the proof is the same as the proof of Corollary \ref{thm::Zbias-binary-case}.
\end{proof}

\begin{proof}
[of Theorem \ref{thm::general-iv-conf-weak}]
First, we consider the treatment effect on the population under treatment. Taking $U=Y(0)$ in Lemma \ref{lemma::chiba2009-1}, we have $\FACE\geq \ACE_1^\true$, because $A\ind Y(0)\mid Y(0)$, $\pr\{ A=1\mid Y(0) \}$ is non-decreasing in $Y(0)$, and $\E\{ Y\mid A=0, Y(0)\}  = Y(0)$ is non-decreasing in $Y(0)$. 
The condition $\cov\{  \Pi, \E( Y\mid A=0, \Pi )   \}  \leq 0$ implies that $\ACE_1^\adj \geq \FACE$ according to Lemma \ref{lemma::adjust-formula}. Therefore, $ \ACE_1^\adj \geq \FACE  \geq   \ACE_1^\true.$

Second, we take $U=Y(1)$ in Lemma \ref{lemma::chiba2009-2}, and by a similar argument as above we have $ \ACE_0^\adj \geq \FACE  \geq  \ACE_0^\true.$

The conclusion holds because $\ACE^\true =   f \ACE_1^\true + (1-f) \ACE_0^\true$ and $\ACE^\adj = f \ACE_1^\adj + (1-f) \ACE_0^\adj.$
\end{proof}

\begin{proof}
[of Theorem \ref{thm::general-iv-conf}]
Under the additive model of $A$ given $\Pi$ and $U = \{ Y(1), Y(0)  \}$, we have the following results.
First, $\pr(A=1\mid \Pi) = \Pi$ is increasing in $\Pi.$
Second, $\Pi \ind \{ Y(1), Y(0)\} $ implies
\begin{eqnarray*}
\pr\{  A=1\mid \Pi, Y(1) = y_1\}  &=& \int \pr(A=1\mid \Pi, U) F(\d y_0 \mid y_1)  \\
&=& \int\{  \Pi + \delta(y_1)  + \eta(y_0)  \}  F(\d y_0 \mid y_1) \\
&=& \Pi + \delta(y_1) + \int  \eta(y_0)   F(\d y_0 \mid y_1) \equiv \Pi  + \widetilde{\delta}(y_1).
\end{eqnarray*}
Denote the infimum and supremum of $Y(0)$ given $Y(1)=y_1$ by $\underline{y}_0(y_1)$ and $\overline{y}_0$, with the later not depending on $y_1$ according to Condition (c) of Theorem \ref{thm::general-iv-conf}. Applying integration or summation by parts, we have
\begin{eqnarray*}
\widetilde{\delta}(y_1) &=& \delta(y_1) + \eta(y_0) F(y_0\mid y_1)|_{y_0 = \underline{y}_0(y_1)}^{y_0 = \overline{y}_0}
- \int  \left\{  \frac{ \d \eta(y_0) }{ \d y_0} \right\} F( y_0 \mid y_1) \d y_0 \\
&=&\delta(y_1) + \eta(\overline{y}_0) -  \int  \left\{  \frac{ \d \eta(y_0) }{ \d y_0} \right\} F( y_0 \mid y_1) \d y_0 .
\end{eqnarray*}
The function $ \widetilde{\delta}(y_1)$ is non-decreasing in $y_1$, because
\begin{eqnarray*}
\frac{ \d~\widetilde{\delta}(y_1) }{ \d  y_1} 
= \frac{\d \delta(y_1) }{\d y_1}
-  \int  \left\{  \frac{ \d \eta(y_0) }{ \d y_0} \right\} \left\{  \frac{ \partial  F( y_0 \mid y_1)  }{\partial y_1 } \right\}  \d y_0
\geq  0.
\end{eqnarray*}

Third, following the same reasoning as the second argument, we have
$
\pr\{  A=1\mid \Pi, Y(1) = y_0\} =  \Pi +    \widetilde{\eta}(y_0),
$
with $\widetilde{\eta}(y_0)$ being a non-decreasing function of $y_0.$
Fourth, $\Pi \ind Y(1)$ implies
$
\pr\{ A=1\mid Y(1) = y_1 \} 
= f +  \widetilde{\delta}(y_1),
$
which is non-decreasing in $y_1.$
Fifth, $\Pi \ind Y(0)$ implies
$
\pr\{ A=1\mid Y(0) = y_0 \}   = f +  \widetilde{\eta}(y_0),
$
which is non-decreasing in $y_0.$

According the fourth and fifth arguments above, Condition (a) in Theorem \ref{thm::general-iv-conf-weak} holds. Therefore, we need only to verify Condition (b) in Theorem \ref{thm::general-iv-conf-weak} to complete the proof. 

We have shown that $\pr\{ A=1\mid \Pi, Y(1) \} = \Pi + \widetilde{\delta} \{  Y(1) \}$, which is additive and non-decreasing in $\Pi$ and $Y(1)$. According to Lemma \ref{lemma::general-interaction}, we know that 
\begin{eqnarray}
\frac{  \partial \pr\{ Y(1) \leq  y_1 \mid A=1, \Pi =  \pi \}  }{ \partial \pi} \geq  0
\label{eq::distributional-1}
\end{eqnarray}
for all $y_1$ and $\pi.$
We have also shown that $\pr\{ A=1\mid \Pi, Y(0) \} = \Pi + \widetilde{\eta} \{  Y(0) \}$, which is additive and non-decreasing in $\Pi$ and $Y(0)$. Again according to Lemma \ref{lemma::general-interaction}, we know that 
\begin{eqnarray}
\frac{  \partial \pr\{ Y(0) \leq  y_0 \mid A=0, \Pi =  \pi \}  }{ \partial \pi} \geq  0
\label{eq::distributional-2}
\end{eqnarray}
for all $y_0$ and $\pi.$
According to \citet{xie2008some}, the above negative distributional associations in \eqref{eq::distributional-1} and \eqref{eq::distributional-2} imply the negative associations in expectation between $Y(0)$ and $\Pi$ given $A$, as required by condition (b) of Theorem \ref{thm::general-iv-conf-weak}.
\end{proof}

\begin{proof}
[of Corollary \ref{thm::general-iv-conf-binary}]
As shown in the proof of Theorem \ref{thm::general-iv-conf}, the conclusion follows immediately from the five ingredients. We will show that they hold even if there is non-negative interaction between binary $Y(1)$ and $Y(0)$. The following proof is in parallel with the proof of Theorem \ref{thm::general-iv-conf}.

First, $\pr(A=1\mid \Pi) = \Pi$ is increasing in $\Pi.$
Second, 
\begin{eqnarray} 
&&\pr\{  A=1\mid \Pi, Y(1) = y_1 \}   \nonumber  \\
&=& E\left[  \pr\{  A=1\mid \Pi, Y(1) = y_1, Y(0) \}  \mid  \Pi, Y(1) = y_1 \right]   \nonumber   \\
&=&
\E \left\{    \alpha +   \Pi + \delta  y_1 + \eta Y(0) + \theta y_1 Y(0) \mid \Pi, Y(1) = y_1    \right\}  \nonumber  \\
&=& \alpha + \Pi + \delta  y_1 + \eta \pr\{  Y(0) = 1\mid Y(1) = y_1\}  
+ \theta y_1 \pr\{  Y(0)=1\mid Y(1) = y_1\}   \label{eq::given-y1} \\
&\equiv & \Pi + \widetilde{\delta}   [  y_1  - \E\{ Y(1) \}  ]  . \label{eq::given-y1-linear}
\end{eqnarray}
The last equation in \eqref{eq::given-y1-linear} follows from the fact that $Y(1)$ is binary and the functional form must be linear in $y_1$, where the coefficient is
\begin{eqnarray}
 \widetilde{\delta}   &=& \pr\{  A=1\mid \Pi, Y(1) = 1 \} - \pr\{  A=1\mid \Pi, Y(1) = 0 \}  \nonumber  \\
 &=& \delta + \eta [  \pr\{  Y(0) = 1\mid Y(1) = 1\}  -  \pr\{  Y(0) = 1\mid Y(1) = 0\} ] 
 +\theta   \pr\{  Y(0)=1\mid Y(1) = 1\}  \nonumber  \\
 && \label{eq::coef-given-y1} \\
 &\geq & \eta [  \pr\{  Y(0) = 1\mid Y(1) = 1\}  -  \pr\{  Y(0) = 1\mid Y(1) = 0\} ], \label{eq::coef-given-y1-bound}
\end{eqnarray}  
where \eqref{eq::coef-given-y1} follows from \eqref{eq::given-y1}, and \eqref{eq::coef-given-y1-bound} follows from $\delta \geq 0$ and $\theta \geq 0.$
Because $\OR_Y\geq 1$, the potential outcomes have non-negative association, implying that their risk difference $\RD_Y =\pr\{  Y(0) = 1\mid Y(1) = 1\}  -  \pr\{  Y(0) = 1\mid Y(1) = 0\}  \geq 0$.
Therefore, $\widetilde{\delta}  \geq 0$, and $\pr\{  A=1\mid \Pi, Y(1)   \}$ is additive and non-decreasing in $\Pi$ and $Y(1)$.

Third, similar to the second argument, we have
$
\pr\{  A=1\mid \Pi, Y(0) = y_0 \} = \Pi + \widetilde{\eta}[  y_0  - \E\{ Y(0) \}  ] 
$
with 
$
 \widetilde{\eta} 
  \geq 0.
$
Therefore, $\pr\{  A=1\mid \Pi, Y(0)   \}$ is additive and non-decreasing in $\Pi$ and $Y(0)$.
Fourth, $\Pi\ind Y(1)$ implies that
$
\pr\{ A=1\mid Y(1) \} = f +  \widetilde{\delta}   Y(1)
$
is increasing in $Y(1).$
Fifth, $\Pi\ind Y(0)$ implies that
$
\pr\{ A=1\mid Y(0) \} = f +  \widetilde{\eta}   Y(0)
$
is increasing in $Y(0).$

With these five ingredients, the rest of the proof is exactly the same as the proof of Theorem \ref{thm::general-iv-conf}.
\end{proof}

\begin{proof}
[of Theorem \ref{thm::general-iv-conf-multi}]
First, $\pr(A=1\mid \Pi) = \Pi$ is non-decreasing in $\Pi$. Second,
$$
\pr\{  A=1\mid \Pi, Y(1)=y_1\} = \Pi \delta(y_1) \int \delta(y_0) F(\d y_0\mid y_1)\equiv \Pi \widetilde{\delta}(y_1)
$$
is multiplicative and non-decreasing in $\Pi$ and $y_1$, following the same argument as the proof of Theorem \ref{thm::general-iv-conf}. Third, $\pr\{  A=1\mid \Pi, Y(0)=y_0\}  = \Pi \widetilde{\eta}(y_0)$ is multiplicative and non-decreasing in $\Pi$ and $y_0$. Fourth, $\pr\{ A=1\mid Y(1)=y_1\} = f \widetilde{\delta}(y_1)$ is non-decreasing in $y_1$. Fifth, $\pr\{ A=1\mid Y(0)=y_0\} = f \widetilde{\eta}(y_0)$ is non-decreasing in $y_0$. 

The multiplicative models and Lemma \ref{lemma::inter-multi-general} imply that for all $\pi, y_1$ and $y_0$,
\begin{eqnarray}
\frac{  \partial \pr\{ Y(1) \leq  y_1 \mid A=1, \Pi =  \pi \}  }{ \partial \pi} =  0\leq 0,\quad
\frac{  \partial \pr\{ Y(0) \leq  y_0 \mid A=0, \Pi =  \pi \}  }{ \partial \pi} \geq  0.
\label{eq::distributional-3}
\end{eqnarray}
The rest part is the same as the proof of Theorem \ref{thm::general-iv-conf}.
\end{proof}

\begin{proof}
[of Corollary \ref{thm::general-iv-conf-binary-multi}]
First, $\pr(A=1\mid \Pi) = \Pi$ is non-decreasing in $\Pi$. Second,
$$
\pr\{  A=1\mid \Pi, Y(1)=y_1\} = \alpha  \Pi \delta^{y_1} \E\{  \eta^{Y(0)} \theta^{y_1Y(0)} \mid Y(1) = y_1 \} \equiv \alpha \Pi \widetilde{\delta}^{Y(1)},
$$
where the functional form must be multiplicative because of binary $Y(0)$, and the parameter $\widetilde{\delta}$ is
\begin{eqnarray*}
\widetilde{\delta} &=&  \frac{  \pr\{  A=1\mid \Pi, Y(1)=1\}  }{   \pr\{  A=1\mid \Pi, Y(1)= 0 \} }  \\
&=& \delta \times    \frac{  \E\{  \eta^{Y(0)} \theta^{Y(0)} \mid Y(1) = 1 \}  }
{     \E\{  \eta^{Y(0)}  \mid Y(1) = 0 \}    } \\
&=& \delta \times   \frac{     \eta \theta \pr\{  Y(0)=1\mid Y(1) = 1  \}  +  \pr\{  Y(0)=0 \mid Y(1) = 1  \}    }
{  \eta  \pr\{  Y(0)=1\mid Y(1) = 0  \}  +  \pr\{  Y(0)=0 \mid Y(1) = 0  \}   } \\
&=& \delta \times    \frac{    ( \eta \theta - 1)  \pr\{  Y(0)=1\mid Y(1) = 1  \}  +  1   }
{  (\eta-1)  \pr\{  Y(0)=1\mid Y(1) = 0  \}  +  1 } .
\end{eqnarray*}
Because $\OR_Y \geq 1$, we have $ \pr\{  Y(0)=1\mid Y(1) = 1  \} \geq \pr\{  Y(0)=1\mid Y(1) = 0  \} $, which implies that $\widetilde{\delta}  \geq 1.$ Therefore, $\pr\{  A=1\mid \Pi, Y(1)\} $ is multiplicative and non-decreasing in $\Pi$ and $Y(1)$. Third, we can similarly show that $\pr\{  A=1\mid \Pi, Y(0)\} $ is multiplicative and non-decreasing in $\Pi$ and $Y(0).$ Fourth, $\pr\{ A=1\mid Y(1)=y_1\} = \alpha f \widetilde{\delta}^{y_1}$ is non-decreasing in $y_1$. Fifth, $\pr\{ A=1\mid Y(0)=y_0\} = \alpha f \widetilde{\eta}^{y_0}$ is non-decreasing in $y_0$.  

The rest part is the same as the proof of Theorem \ref{thm::general-iv-conf-multi}.
\end{proof}

\begin{proof}
[of Theorem \ref{thm::Zbias-scalar-case-general-direct}.]
In Figure \ref{fg::DAG-direct}, $Z$ and $U$ are two independent confounders for the relationship between $A$ and $Y$. Because $\pr(A=1\mid Z=z, U=u)$ and $\E(Y\mid A=a, Z=z, U=u)$ are non-decreasing in $z$ and $u$ for both $a=0$ and $1$, Lemmas \ref{lemma::vw2008}--\ref{lemma::chiba2009-2} imply that the unadjusted estimator, $\FACE$, is larger than or equal to $\ACE^\true, \ACE_1^\true$ and $\ACE_0^\true$.

The independence between $Z$ and $U$ implies $\pr(A=1\mid Z=z) = \int \pr(A=1\mid Z=z, U=u) F(\d u)$, and the monotonicity of $\pr(A=1\mid Z=z, U=u)$ in $z$ implies that $\pr(A=1\mid Z=z)$ is non-decreasing in $z$. The rest of the proof is identical to the proof of Theorem \ref{thm::Zbias-scalar-case-general}.
\end{proof}

\section*{Appendix 3.\quad  Extensions to Other Causal Measures}
\label{sec::extension}

\subsection*{Appendix 3$\cdot$1.\quad  Distributional Causal Effects}

Sometimes we are also interested in estimating the distributional causal effects \citep{ju2010criteria} for the treatment, control and whole populations:
\begin{eqnarray*}
\DCE_1^\true(y) &=& \pr\{  Y(1)>y\mid A=1 \} - \pr\{Y(0) >y\mid A=1 \} ,\\
\DCE_0^\true(y) &=& \pr\{  Y(1)>y\mid A=0 \} - \pr\{Y(0) >y\mid A=0 \} ,\\
\DCE^\true(y)     &=& \pr\{  Y(1)>y \} - \pr\{Y(0) >y \} .
\end{eqnarray*} 

The unadjusted estimator is
$$
\FDCE(y) = \pr(Y>y\mid A=1) - \pr(Y>y\mid A=0). 
$$
The adjusted estimators for the treatment, control and whole populations are
\begin{eqnarray*}
\DCE_1^\adj(y) &=& \pr(Y>y\mid A=1) - \int  \pr(Y>y\mid A=0, z) F(\d z\mid A=1),\\
\DCE_0^\adj(y) &=& \int \pr(Y>y\mid A=1, z) F(\d z\mid A=0) - \pr(Y>y\mid A=0),\\
\DCE^\adj(y)     &=&  \int \pr(Y>y\mid A=1, z) F(\d z) - \int  \pr(Y>y\mid A=0, z) F(\d z) .
\end{eqnarray*}

If the outcome is binary, then the distributional causal effects at $y<1$ are the average causal effects, and zero at $y\geq 1$. All results about distributional causal effects reduce to average causal effects for binary outcome. For a general outcome, the distributional causal effects are the average causal effects on the dichotomized outcome $I_y = I(Y>y).$ Therefore, if we replace the outcome $Y$ by $I_y$ in Theorems \ref{thm::Zbias-scalar-case-general}--\ref{thm::Zbias-scalar-case-multi}, the results about Z-Bias hold for distributional effects. For instance, the condition that $\pr(Y>y\mid A=a, U=u)$ is non-decreasing in $u$ for all $a$ is the same as requiring a non-negative sign on the arrow $U\rightarrow Y$, according to the theory of signed directed acyclic graphs \citep{vanderweele2010signed}. The following theorem states the results analogous to Theorems \ref{thm::general-iv-conf-weak}--\ref{thm::general-iv-conf-multi}.

\begin{corollary}
\label{thm::DCE-Zbias}
In the causal diagram of Figure \ref{fg::DAG2}, if for all $y$ and for both $a=1$ and $0$, 
\begin{enumerate}
[(a)]
\item
$\pr\{  Y(a) > y\mid A=1  \} \geq \pr\{  Y(a) >y\mid A=0 \}$; 

\item
$\cov\{ \Pi, \pr(Y>y\mid A=a, \Pi) \} \leq 0$;
\end{enumerate}
then  
\begin{eqnarray}
\label{eq::result-DCE}
\begin{pmatrix}
\DCE_1^\adj (y) \\
\DCE_0^\adj (y) \\
\DCE^\adj (y)
\end{pmatrix}
\geq 
\begin{pmatrix}
\FDCE(y)  \\
\FDCE(y)  \\
\FDCE(y) 
\end{pmatrix}
\geq 
\begin{pmatrix}
\DCE_1^\true(y) \\
\DCE_0^\true(y) \\
\DCE^\true (y)
\end{pmatrix}.
\end{eqnarray}
Under the conditions of Theorems \ref{thm::general-iv-conf} and \ref{thm::general-iv-conf-multi}, \eqref{eq::result-DCE} holds. 
\end{corollary}

\begin{proof}
[of Corollary \ref{thm::DCE-Zbias}]
Condition (a) of Corollary \ref{thm::DCE-Zbias} is equivalent to $\pr\{  A=1\mid I_y(a)=1 \} \geq \pr\{ A=1\mid I_y(a) = 0\}$, and Condition (b) of Corollary \ref{thm::DCE-Zbias} is equivalent to $\cov\{ \Pi, \E(I_y\mid A=a, \Pi) \} \leq 0$. Therefore, the conclusion follows from Theorem \ref{thm::general-iv-conf-weak}. 

According to the proofs of Theorems \ref{thm::general-iv-conf} and \ref{thm::general-iv-conf-multi}, we have 
\begin{eqnarray*}
&&\pr\{  A=1\mid I_y(a)=1 \} = \pr\{  A=1\mid Y(a) > y \}   \geq \pr\{ A=1\mid Y(a) = y \} \\
&\geq&  \pr\{  A=1\mid Y(a) \leq  y \}  =   \pr\{  A=1\mid I_y(a) = 0\},
\end{eqnarray*}
because of monotonicity of $ \pr\{ A=1\mid Y(a) \}$ in $Y(a)$. Therefore, Condition (a) of Theorem \ref{thm::DCE-Zbias} holds. Under the conditions of Theorems \ref{thm::general-iv-conf} and \ref{thm::general-iv-conf-multi}, we have also shown in \eqref{eq::distributional-1}--\eqref{eq::distributional-3} that for all $a, y$ and $\pi$,
$
   \partial \pr( Y \leq  y \mid A=a, \Pi =  \pi )    / \partial \pi  \geq  0,
$
which implies that $ \E(   I_y \mid A=a, \Pi =  \pi ) $ is non-increasing in $\pi$. Therefore, Condition (b) of Theorem \ref{thm::DCE-Zbias} holds. The proof is complete.
\end{proof}

\subsection*{Appendix 3$\cdot$2.\quad Ratio Measures}

In many applications with binary or positive outcomes, we are also interested in assessing causal effects on the ratio scale for the treatment, control and whole populations, defined as
\begin{eqnarray*}
\RRtrue_1 = \frac{\E\{ Y(1) \mid A=1 \}}{ \E\{ Y(0) \mid A=1 \}} ,\quad 
\RRtrue_0 = \frac{\E\{ Y(1)  \mid A=0 \}}{ \E\{ Y(0)  \mid A=0 \}} ,\quad 
\RRtrue = \frac{\E\{ Y(1) \}}{ \E\{ Y(0)  \}} .
\end{eqnarray*}

The unadjusted estimator on the ratio scale is
$$
\RRunadj = \frac{\E( Y \mid A=1 ) }{ \E(  Y \mid A=0 ) } . 
$$
The adjusted estimators on the ratio scale for the treatment, control and whole populations are
\begin{eqnarray*}
\RRadj_1 &=&  \frac{\E(  Y \mid A=1 ) }{  \int  \E\{ Y  \mid A=0, Z=z \} F(\d z\mid A=1) } ,\\
\RRadj_0 &=& \frac{  \int  \E\{ Y  \mid A=1, Z=z \} F(\d z\mid A=0)  }{ \E(  Y \mid A=0 ) } ,\\
  \RRadj &=& \frac{  \int  \E\{ Y  \mid A=1, Z=z \} F(\d z )  }{ \int  \E\{ Y  \mid A=0, Z=z \} F(\d z )}  .
\end{eqnarray*}
With a general instrumental variable $Z$, we can replace $Z$ by $\Pi$ in the definitions of the adjusted estimators.

\begin{corollary}
\label{thm::rr}
All the theorems and corollaries in \S\S \ref{sec::scalar} and \ref{sec::general} hold on the ratio scale, i.e., under their conditions,
\begin{eqnarray*}
\begin{pmatrix}
\RRadj_1 \\
\RRadj_0 \\
\RRadj
\end{pmatrix}
\geq 
\begin{pmatrix}
\RRunadj  \\
\RRunadj  \\
\RRunadj 
\end{pmatrix}
\geq 
\begin{pmatrix}
\RRtrue_1\\
\RRtrue_0\\
\RRtrue
\end{pmatrix}.
\end{eqnarray*}
\end{corollary}

\begin{proof}
[of Corollary \ref{thm::rr}]
First, $\RRtrue$ is a convex combination of $\RRtrue_1 $ and $ \RRtrue_0$, and $\RRadj$ is a convex combination of $\RRadj_1$ and $\RRadj_0$, which are formally stated in \citet[][eAppendix]{Ding::2015}. Then the conclusion follows from the proofs of the theorems above.
\end{proof}

\subsection*{Appendix 3$\cdot$3.\quad  Average Over Observed Covariates}

In practice, we need to adjust for the observed covariates $X$ that are confounders affecting both the treatment and outcome. The discussion in previous sections is conditional on or within strata of observed covariates $X$, and the causal effects and their estimators are given $X$. For example, 
\begin{eqnarray*}
\ACE^\true(x) &=& \E\{ Y(1)\mid X=x \} - \E\{ Y(0)\mid X=x \},\\
\FACE(x) &=& \E(Y\mid A=1, X=x) - \E(Y\mid A=0, X=x),\\
\ACE^\adj(x)&=&\int \E(Y\mid A=1, Z=z, X=x) F(\d z\mid X=x)  \\
&&- \int \E(Y\mid A=0, Z=z, X=x) F(\d z\mid X=x) ,
\end{eqnarray*}
and other conditional quantities can be analogously defined. If the conditions in the theorems and corollaries in \S\S \ref{sec::scalar} and \ref{sec::general} hold within each level of $X$, then the conclusions in \eqref{eq::result} and \eqref{eq::result-DCE} hold not only within each level of $X$ but also averaged over $X$. For example, for the average causal effects, we have
\begin{eqnarray*}
\begin{pmatrix}
\int \ACE^\adj_1(x) F(\d x\mid A=1) \\
\int \ACE^\adj_0(x) F(\d x\mid A=0)  \\
\int \ACE^\adj (x) F(\d x ) 
\end{pmatrix}
&\geq& 
\begin{pmatrix}
\int \FACE(x) F(\d x\mid A=1) \\
\int \FACE(x) F(\d x\mid A=0)  \\
\int \FACE (x) F(\d x ) 
\end{pmatrix} \\
&\geq& 
\begin{pmatrix}
\int \ACE^\true_1(x) F(\d x\mid A=1) \\
\int \ACE^\true_0(x) F(\d x\mid A=0)  \\
\int \ACE^\true (x) F(\d x ) 
\end{pmatrix} .
\end{eqnarray*}

\end{document}